\documentclass[12pt,regno]{amsart}
\usepackage{epsf,epsfig,amsmath}
\usepackage{amssymb,latexsym}

\usepackage{amsthm} 
\usepackage{amsfonts}
\usepackage{graphicx,psfrag}

\usepackage{tikz}
\usepackage{tikz-cd}

\numberwithin{equation}{section}

\newtheorem{thm}{Theorem}[section]
\newtheorem{pro}[thm]{Proposition}
\newtheorem{lem}[thm]{Lemma}
\newtheorem{con}[thm]{Conjecture}

\newtheorem{rem}[thm]{Remark}
\newtheorem{defi}[thm]{Definition}

\newtheorem*{cor*}{Corollary}
\newtheorem*{thm*}{Theorem}

\title[Combinatorial invariance conjecture]{The combinatorial invariance conjecture for parabolic Kazhdan--Lusztig polynomials of lower intervals}

\author{Mario Marietti}

\address{Dipartimento  di Ingegneria Industriale e Scienze Matematiche, Universit\`a Politecnica delle Marche, Via Brecce Bianche, 60131 Ancona,  Italy}

\email{m.marietti@univpm.it}


\begin{document}

\begin{abstract}
The aim of this work is to prove a conjecture related to  the Combinatorial Invariance Conjecture of  Kazhdan--Lusztig polynomials, in the parabolic setting, for lower intervals in every arbitrary Coxeter group. This result improves and generalizes, among other results,  the main results of [Advances in Math. {202} (2006), 555-601], [Trans. Amer. Math. Soc. {368} (2016), no. 7, 5247--5269]. 
\end{abstract}

\maketitle

\section{Introduction}
Kazhdan--Lusztig polynomials play a central role in Lie theory and representation theory. They are polynomials $P_{u,v}(q)$, in one variable $q$, which are associated to  pairs of elements $u,v$ in a Coxeter group $W$. 
They were defined  by Kazhdan and Lusztig in \cite{K-L} in order to introduce the (now called)  
Kazhdan--Lusztig representations of the Hecke algebra of $W$, and soon have found applications in many other contexts. 

Among others, the combinatorial aspects of Kazhdan--Lusztig polynomials have received much attention from the start, and are still a fascinating field of research. Recently, Elias and  Williamson \cite{E-W} proved the long-standing conjecture about the nonnegativity of  the coefficients of  Kazhdan--Lusztig polynomials of all Coxeter groups, thus generalizing the analogous result by Kazhdan and Lusztig  on finite and affine Weyl groups appearing in \cite{KL2}, where $P_{u,v}(q)$ is shown to be the Poincar\'e polynomial of the local intersection cohomology groups  of the Schubert variety associated with $v$ at any point of the Schubert variety associated with $u$ (in the full flag variety).

At present, from a combinatorial point of view, the most intriguing conjecture about Kazhdan--Lusztig polynomials  is arguably what is usually referred to as the Combinatorial Invariance Conjecture of Kazhdan--Lusztig polynomials. It was independently formulated by Lusztig in private and by Dyer in \cite{Dyeth}.
\begin
{con}
\label{comb-inv-con}
The Kazhdan--Lusztig polynomial $P_{u,v}(q)$ depends only on the isomorphism class of the interval $[u,v]$ as a poset. 
\end{con}
The Combinatorial Invariance Conjecture of  Kazhdan--Lusztig polynomials is equivalent to the analogous conjecture on the combinatorial invariance of  Kazhdan--Lusztig $R$-polynomials. These  also are polynomials $R_{u,v}(q)$   indexed by a pair of elements $u,v$ in $W$ and were introduced by Kazhdan--Lusztig in the same article \cite{K-L}. The Kazhdan--Lusztig $R$-polynomials are equivalent to  the Kazhdan--Lusztig polynomials of $W$ (in a precise sense, see Remark~\ref{senso}).

 \bigskip
 
 In  \cite{Deo87}, for any choice of a subset $H\subseteq S$, Deodhar introduces two modules of the Hecke algebra of $W$, two parabolic analogues $\{P_{u,w}^{H,x} (q)\}_{u,w\in W^H}$ of the Kazhdan--Lusztig polynomials, and two parabolic analogues $\{R_{u,w}^{H,x} (q)\}_{u,w\in W^H}$ of the Kazhdan--Lusztig $R$-polynomials, one for $x=q$ and one for $x=-1$.  The parabolic Kazhdan--Lusztig and $R$-polynomials have deep algebraic and geometric significance; they are indexed by pairs of elements in the set $W^H$ of minimal coset representatives with respect to the standard parabolic subgroup $W_H$, and play, in the parabolic setting, a role that is parallel to the role that the ordinary Kazhdan--Lusztig and $R$-polynomials play in the ordinary setting. Moreover, they generalize the ordinary Kazhdan--Lusztig and $R$-polynomials since these are obtained  in the special trivial case when $H= \emptyset$ (for both $x=q$ and $-1)$. As in the ordinary case, the family of parabolic Kazhdan--Lusztig polynomials is equivalent to the family of parabolic Kazhdan--Lusztig $R$-polynomials.
 
 The problem of the combinatorial invariance of parabolic Kazhdan--Lusztig polynomials, which is stronger than the combinatorial invariance of the ordinary Kazhdan--Lusztig polynomials, has also attracted much attention (see, for instance, \cite{Bpac} and \cite{Btrans}).
 Only recently, however, the statement one gets by replacing the ordinary interval with the parabolic interval in Conjecture~\ref{comb-inv-con} has been found to be false (see \cite{BMS} and \cite{Mtrans} for counterexamples).
In \cite{Mtrans}, it is proposed that  the right approach to the generalization of  Conjecture~\ref{comb-inv-con}  to the parabolic setting could be studying to what extent the following conjecture is true.
\begin{con}
\label{comb-inv-con-parab}
Let $(W_1,S_1)$ and $(W_2,S_2)$ be two Coxeter systems, $H_1 \subseteq S_1$ and 
$H_2\subseteq S_2$. Let $u_1,v_1 \in W_1^{H_1}$ and $u_2,v_2 \in W_2^{H_2}$ be such that there exists 
a poset-isomorphism  from $[u_1,v_1]$ to $[u_2,v_2]$ that restricts to a 
poset-isomorphism from $[u_1,v_1]^{H_1}$ to $[u_2,v_2]^{H_2}$. Then $P_{u_1,v_1}^{H_1,x} (q)=P_{u_2,v_2}^{H_2,x} (q)$  (equivalently, $R_{u_1,v_1}^{H_1,x} (q)=R_{u_2,v_2}^{H_2,x} (q)$).
\end{con}  
Clearly, Conjecture~\ref{comb-inv-con-parab} reduces to Conjecture~\ref{comb-inv-con} for $H_1=H_2=\emptyset$.  

Conjecture~\ref{comb-inv-con} and Conjecture~\ref{comb-inv-con-parab}, if true, would have  interesting implications in the many contexts where ordinary and parabolic Kazhdan--Lusztig polynomials have applications. Among them, one of the most fascinating and (according to many experts in the field) surprising consequences would be in the topology of Schubert varieties of full and partial flag varieties.  For the full flag variety, we refer the reader to the discussion in \cite[\S 3]{Bslc}. For its generalization to the partial flag variety, the reader should have in mind the results by Kashiwara and Tanisaki \cite{KT} showing the role  of the parabolic Kazhdan--Lusztig polynomials for the Schubert varieties of the partial flag variety.

In \cite{Mtrans}, Conjecture~\ref{comb-inv-con-parab} is proved to hold true  for lower intervals  (that is, when $u_1$ and $u_2$ are the identity elements), in  the case of doubly laced Coxeter groups (and in the case of dihedral Coxeter groups, which is much easier). The aim of this work is to prove the following more general result.
\begin{thm}
\label{elena}
Conjecture~\ref{comb-inv-con-parab} holds true for all lower intervals in  every arbitrary Coxeter group.  
\end{thm}
(Another new piece of evidence in favor of Conjecture~\ref{comb-inv-con-parab} was recently given by Brenti in \cite{Bnuovo}).

Indeed, we prove the following slightly more general result.
\begin{thm}
\label{congettura}
Let $(W_1,S_1)$ and $(W_2,S_2)$ be two arbitrary Coxeter systems, with identity elements $e_1$ and $e_2$, and let $H_1 \subseteq S_1$ and 
$H_2\subseteq S_2$. Let $v_1 \in W_1^{H_1}$ and $v_2 \in W_2^{H_2}$ be such that there exists 
a poset-isomorphism  $\psi$ from $[e_1,v_1]$ to $[e_2,v_2]$ that restricts to a 
poset-isomorphism from $[e_1,v_1]^{H_1}$ to $[e_2,v_2]^{H_2}$. Then, for all $u,w \in [e_1,v_1]^{H_1}$, we have 
$$P_{u,w}^{H_1,x} (q)=P_{\psi(u),\psi(w)}^{H_2,x} (q)   \quad \text{ and } \quad R_{u,w}^{H_1,x} (q)=R_{\psi(u),\psi(w)}^{H_2,x} (q).$$
\end{thm}
As a corollary, the parabolic Kazhdan--Lusztig polynomial $ P_{u,w}^{H,x} (q)$ and $R$-polynomial $ R_{u,w}^{H,x} (q)$ are determined by the isomorphism class of the interval $[e,w]$ and by how the parabolic interval $[e,w]^H=[e,w]\cap W^H$ embeds in $[e,w]$.

Theorem \ref{congettura} is proved by providing an explicit method to compute the parabolic Kazhdan--Lusztig $R$-polynomials $R_{u,w}^{H,x}(q)$ (and so also the parabolic Kazhdan--Lusztig $P_{u,w}^{H,x}(q)$).
This method is based on the concept of an \emph{$H$-special matching} introduced in \cite{Mtrans}: an $H$-special matching of $w$ is an involution
\( M:[e,w]\rightarrow [e,w] \) such that 
\begin{enumerate}
\item either \( u \lhd M(u)\) or  \( u \rhd M(u)\), for all \( u\in [e,w]  \),
\item  if $u_1\lhd u_2$ then  $M(u_1)\leq M(u_2),$
 for all \( u_1,u_2\in [e,w] \) such that \( M(u_1)\neq u_2 \),
 \item if $u \leq w$, $u \in W^H$, and $M(u) \lhd u$, then  $M(u) \in W^H.$
\end{enumerate}
(We denote by $\leq$ the Bruhat order and write $x\lhd y$ to mean that 
$x$ is an immediate predecessor of $y$).

The set of all $H$-special matchings of $w$ depends  only on the isomorphism class  of the interval $[e,w]$ and on how the parabolic interval $[e,w]^{H}$ embeds in $[e,w]$. 
We prove  that $H$-special matchings may be used in place of 
left multiplications in the recurrence formula that computes the parabolic Kazhdan--Lusztig $R$-polynomials.
\begin{thm}
\label{computa}
If $M$ is an $H$-special matching of $w$, then the parabolic Kazhdan--Lusztig $R$-polynomial $R_{u,w}(q)$ satisfies: 
\begin{equation}
 \label{calcola}
 R_{u,w}^{H,x} (q)= \left\{ \begin{array}{ll}
R_{M(u),M(w)}^{H,x}(q), & \mbox{if $M(u)  \lhd u$,} \\
(q-1)R_{u,M(w)}^{H,x}(q)+qR_{M(u),M(w)}^{H,x}(q), & \mbox{if $M(u) \rhd u$
and $M(u) \in W^{H}$,} \\
(q-1-x)R_{u,M(w)}^{H,x}(q), & \mbox{if $M(u) \rhd u$ and  $M(u) \notin W^{H}$.} 
\end{array} \right. 
\end{equation}
\end{thm}
Theorem \ref{computa} directly implies  Theorem \ref{congettura}. Indeed, suppose the hypotheses of Theorem \ref{congettura} are fulfilled: thus $M$ is an $H_1$-special matching of $w$ if and only if  $M'=\psi \circ M\circ \psi^{-1}$ is an $H_2$-special matching of $\psi(w)$. We choose such a matching $M$ and apply (\ref{calcola}) to  both $M$ and $M'$: in both computations, we fall in the same case. By iteration, we get the assertion of Theorem \ref{congettura}.

Theorems \ref{elena}, \ref{congettura} and \ref{computa} improve and generalize several results in the literature such as, for instance, the main results of \cite{BCM1}, \cite{DC}, \cite{Mtrans}, \cite{Tel}.

\bigskip 
Since a special matching $M$ is uniquely determined by its action on the dihedral intervals containing the Coxeter generator $M(e)$, special matchings of doubly laced Coxeter groups  are more easily controlled than special matchings of arbitrary Coxeter groups. 
Therefore, a deeper analysis on parabolic Kazhdan--Lusztig $R$-polynomials is needed to prove the result for arbitrary Coxeter groups.
Indeed, we use several new identities 
among which, in particular, certain relations relating different parabolic Kazhdan--Lusztig $R$-polynomials indexed by elements in the same coset of dihedral standard parabolic subgroups.

The rest of the paper is devoted to the proof of   Theorem \ref{computa}.

\bigskip

\section{Notation, definitions and preliminaries}

This section reviews the background material that is needed  in the rest of this work.
We  follow   \cite{BB} and \cite[Chapter 3]{StaEC1} for undefined notation and 
terminology concerning, respectively, Coxeter groups and  partially ordered sets.

\subsection{Coxeter groups}
We fix our notation on a Coxeter system $(W,S)$ in the following list:
\smallskip 

{\renewcommand{\arraystretch}{1.2}
$
\begin{array}{@{\hskip-1.3pt}l@{\qquad}l}
m_{s,t} &  \textrm{the entry of the Coxeter matrix of  $(W,S)$ in position $(s,t)\in S\times S$}, 
\\
e &  \textrm{identity of $W$}, 
\\
\ell  &  \textrm{the length function of $(W,S)$},
\\
T &= \{ w s w ^{-1} : w \in W, \; s \in S \},  \textrm{  the set of {\em reflections} of $W$},
\\
D_R(w) & =\{ s \in S : \; \ell(w  s) < \ell(w ) \},  \textrm{ the right descent set of $w\in W$},
\\
D_L(w) & =\{ s \in S : \; \ell( sw) < \ell(w ) \},  \textrm{ the left descent set of $w\in W$},
\\
W_J & \textrm{ the parabolic subgroup of $W$ generated by $J\subseteq S$},
\\
W^J &=\{ w \in W \, : \; D_{R}(w)\subseteq S\setminus J \}, \textrm{ the set of minimal left coset representatives},
\\
^{J} W &=\{ w \in W \, : \; D_{L}(w)\subseteq S\setminus J \}, 
\textrm{ the set of minimal right coset representatives},
\\
\leq & \textrm{ Bruhat order on $W$ (as well as usual order on $\mathbb R$)},
\\
\textrm{$[u,v]$} & =\{ w \in W \, : \; u \leq w \leq v \}, \textrm{ the (Bruhat) interval generated by $u,v\in W$},
\\
w_0(J) &  \textrm{ the unique maximal  element of $[e,w]\cap W_{J}$, for $J\subseteq  S$},
\\
w_0(s,t) & =w_0(\{s,t\}),  \textrm{  for $s,t\in S$},
\\
\textrm{ $[u,v]^{H}$} &  = \{ z \in W^{H}: \;  u \leq z \leq v \}, \textrm{ the parabolic interval generated by $u,v\in W^H$}. 

\end{array}$
\bigskip

Given $u,v\in W$, we write $u\cdot v$ instead of simply $uv$ when $\ell(uv)=\ell(u)+\ell(v)$ and we want to stress this additivity. On the other hand, when we write $uv$, $\ell(uv)$ can be either $ \ell(u)+\ell(v)$ or smaller.
We make use of the symbol ``$\textrm{-}$''  to separate letters in a word in the alphabet $S$ when we want to stress the fact that we are considering the word rather than the element such word represents.

If $w\in W$, then a {\em reduced expression} for $w$ is a word $s_1\textrm{-}s_2\textrm{-}\cdots \textrm{-}s_q$ such that $w=s_1s_2\cdots s_q$ and $\ell(w)=q$. When no confusion arises, we also write that $s_1s_2\cdots s_q$ is a reduced expression for $w$.

The {\em Bruhat graph}
of $W$ (see \cite{Dye2}, or, e.g., \cite[\S 2.1]{BB} or \cite[\S 8.6]{Hum}) is the directed graph
having $W$ as vertex set and having a directed edge from $u$ to $v$ if and only if $u^{-1}v \in T$ and $\ell (u)<\ell (v)$.  
The {\em Bruhat order} (see, e.g.,  \cite[\S 2.1]{BB} or \cite[\S 5.9]{Hum}), sometimes also called {\em Bruhat-Chevalley order},  is the partial order $\leq $ on $W$ given by the transitive closure of the Bruhat graph of $W$

The following  well-known characterization of Bruhat order
is usually referred to as the {\em Subword Property} (see \cite[\S 2.2]{BB} or \cite[\S 5.10]{Hum}), and is used repeatedly in the following sections,
often without explicit mention. 
By a {\em subword} of a word $s_{1}\textrm{-}s_{2} \textrm{-} \cdots \textrm{-} s_{q}$, we mean
a word of the form
$s_{i_{1}}\textrm{-} s_{i_{2}}\textrm{-} \cdots \textrm{-} s_{i_{k}}$, where $1 \leq i_{1}< \cdots
< i_{k} \leq q$.
\begin{thm}[Subword Property]
\label{subword}
Let $u,w \in W$. The following are equivalent:
\begin{itemize}
\item $u \leq w$ in the Bruhat order,
\item  every reduced expression for $w$ has a subword that is 
a reduced expression for $u$,
\item there exists a  reduced expression for $w$ having a subword that is 
a reduced expression for $u$.
\end{itemize}
\end{thm}

The following results are well known (see, e.g.,  \cite[Theorem~1.1]{Deo77},  \cite[Proposition~2.2.7]{BB} or \cite[Proposition~5.9]{Hum} for the first one,  \cite[\S 2.4]{BB} or \cite[\S 1.10]{Hum} for the second one, and \cite[Lemma~7]{Hom74} for the third one).
\begin{lem}[Lifting Property]
\label{ll}
Let $s\in S$ and $u,w\in W$, $u\leq w$.
\begin{enumerate}
\item[-] If $s\in D_R(w)$ and $s\in D_R(u)$, then $us\leq ws$.
\item[-] If $s\notin D_R(w)$ and $s \notin D_R(u)$, then $us\leq ws$.
\item[-] If $s\in D_R(w)$ and $s\notin D_R(u)$, then $us\leq w$ and $u\leq ws$.
\end{enumerate}
Symmetrically, left versions of the three statements hold.
\end{lem}
\begin{pro}
\label{fattorizzo}
Let $J \subseteq S$. 
\begin{enumerate}
\item[(i)] 
Every $w \in W$ has a unique factorization $w=w^{J} \cdot w_{J}$ 
with $w^{J} \in W^{J}$ and $w_{J} \in W_{J}$; for this factorization, $\ell(w)=\ell(w^{J})+\ell(w_{J})$.
\item[(ii)] Every $w \in W$ has a unique factorization $w=\, _{J} w\,  \cdot \,  ^{J}\! w$ 
with $_{J} w \in W_{J}$, $^{J} \! w \in \,  ^{J} W$; for this factorization, $\ell(w)=\ell(_{J} w )+\ell(^{J} \! w)$.
\end{enumerate}
\end{pro}

\begin{pro}
\label{unicomax}
Let \( J\subseteq S \) and $w \in W$.  The set 
$ W_{J}\cap [e,w] $ has a unique  maximal element  \( w_0(J) \),
so that \( W_{J}\cap [e,w]$ is the interval $[e,w_0(J)] \). 
\end{pro}

Note that, by the uniqueness of the factorizations of Proposition~\ref{fattorizzo}, if 
 $J\subseteq S$ and $w\in W$, 
 then
\begin{eqnarray}
\label{0}
l \in D_L(\,_Jw) & \iff &l \in D_L(w) \cap J.
\end{eqnarray}

Furthermore, it is well known (and immediate to prove) that
$ v\leq w$  implies  both $v^J\leq w^J$  and  ${^J\! v}\leq \,{^J\!w}$.

\subsection{Special matchings}
Let $P$ be a partially ordered set. An element $y\in P$ {\em covers} $x\in P$  if the interval $[x,y]$ coincides with $\{x,y\}$; in this case, we write $x \lhd y$ as well as $y \rhd x$.
The poset $P$ is {\em graded} if $P$ has a minimum and there is a function
$\rho : P \rightarrow {\mathbb N}$ (the {\em rank function}
of $P$) such that $\rho (\hat{0})=0$ and $\rho (y) =\rho (x)
+1$ for all $x,y \in P$ with $x \lhd y$. 
(This definition is slightly different from the one given in \cite{StaEC1}, but is
more convenient for our purposes.) 
The {\em Hasse diagram} of $P$ is the graph having $P$ as vertex set and $ \{ \{ x,y \} \in \binom {P}{2} : \text{ either $x \lhd y$
 or $y \lhd x$} \}$ as edge set.

A {\em matching} of a poset \( P \) is an involution
\( M:P\rightarrow P \) such that \( \{v,M(v)\}\) is an edge in the Hasse diagram of $P$, for all \( v\in V \).
A matching \( M \) of \( P \) is {\em special} if\[
u\lhd v\Longrightarrow M(u)\leq M(v),\]
 for all \( u,v\in P \) such that \( M(u)\neq v \).

\bigskip

Now, let $(W,S)$ be a Coxeter system and recall that  $W$ is a graded partially ordered set (under Bruhat order) having $\ell$ as its rank function.
Given $w\in W$, we say that $M$ is a matching of $w$ if $M$ is a matching of the lower Bruhat interval $[e,w]$.
For \( s\in D_{R}(w) \), we have  a matching 
\(\rho_{s}  \) of $w$  
defined by   \( \rho_{s}(u)=us \), for all $u \in [e,w]$. Symmetrically,  for \( s\in D_{L}(w)$, we have  a matching 
$\lambda_{s}$ of $w$ 
defined by    $\lambda_{s}(u)=su$, for all $u \in [e,w]$. By
the Lifting Property (Lemma~\ref{ll}),  such $\rho_s$ and $\lambda_s$ are special matchings of $w$. 
We call these matchings, respectively,  \emph{right} and \emph{left multiplication
matchings}.

The following two results are used  several times in what follows: the first directly follows from   \cite[Lemma~4.3]{BCM1}, the second  is  \cite[Proposition~5.3]{BCM1}.
We call an interval $[u, v]$ in a poset $P$ {\em dihedral} if it is isomorphic to an interval in a  Coxeter system of rank 2 ordered by Bruhat order. Moreover, given two matchings $M$ and $N$, we say that $M$ and $N$ commute on $X$ if the two compositions $M \circ N (x)$ and $N \circ M (x)$ are defined and equal, for all $x\in X$. We say that two matchings of $w$ commute if they commute everywhere on $[e,w]$.
\begin{lem}
\label{commutano}
Let $w\in W$.  Two special matchings $M$ and $N$ of  $w$ commute if and only if they commute on the lower dihedral intervals  of $[e,w]$ containing $M(e)$ and $N(e)$.
\end{lem}
\begin{lem}
\label{5.3}
Let $J\subseteq S$, $w\in W$, and $M$ be a special matching of $w$. If $M(e)\in J$, then $M$  stabilizes  $[e,w_0(J)]$.
\end{lem}
 In particular,  given two special matchings $M$ and $N$ of $w$ such that  $M(e)\neq N(e)$, we have that $M$ and $N$ commute if and only if they commute on the unique lower dihedral interval $[e, w_0(M(e),N(e))]$, and this lower dihedral interval is stabilized by both $M$ and $N$. 

\bigskip

The following definitions are taken from \cite{Mtrans}.
\begin{defi}
A  right system for $w\in W$ is a quadruple $\mathcal R=(J,s,t,M_{st})$ such that:
\begin{enumerate}

\item[R1.] $J\subseteq S$, $s\in J$, $t\in S\setminus J$, and $M_{st}$ is a special matching of $w_0(s,t)$ such that  $M_{st}(e)=s$ and  $M_{st}(t)=ts$;

\item[R2.] $(u^{J})^{\{s,t\}}\, \cdot \, M_{st} \Big ((u^{J})_{\{s,t\}} \, \cdot \, _{\{s\}} (u_{J})\Big )\, \cdot \,
 ^{\{s\}}(u_{J}) \leq w$,  for all $u\leq w$;

 \item[R3.] 
if $r\in J$ and $r \leq w^J$, then $r$ and $s$ commute;

\item[R4.]
\label{ddddxxxx}
\begin{enumerate}
\item  if $s\leq (w^J)^{\{s,t\}}$ and  $t\leq (w^J)^{\{s,t\}}$, then $M_{st}= \rho_s$,
\item   if $s\leq (w^J)^{\{s,t\}}$ and  $t\not \leq (w^J)^{\{s,t\}}$, then $M_{st}$ commutes with $\lambda_s$,
\item   if $s\not\leq (w^J)^{\{s,t\}}$ and  $t \leq (w^J)^{\{s,t\}}$, then $M_{st}$ commutes with $\lambda_t$;
\end{enumerate}

\item[R5.] if $s\leq \, ^{\{s\}} (w_{J})$, then $M_{st}$ commutes with $\rho_s$ on $[e,w_0(s,t)]$.

\end{enumerate} 

\end{defi}

\begin{defi}
A left system for $w\in W$ is a quadruple $\mathcal L= (J,s,t,M_{st})$ such that:

\begin{enumerate}
\item[L1.] $J\subseteq S$, $s\in J$, $t\in S\setminus J$, and $M_{st}$  is a special matching of $w_0(s,t)$ such that $M_{st}(e)=s$ and  $M_{st}(t)=st$;

\item[L2.] $    (_Ju)^{\{s\}} \, \cdot \, M_{st} \Big( \, (_Ju)_{\{s\}}  \,  \cdot  \,  _{\{s,t\}} (^J  u) \Big) \, 
\cdot \, ^{\{s,t\}}(^J u)    \leq w$,  for all $u\leq w$;
 \item[L3.] 
if $r\in J$ and $r \leq \, ^Jw$, then $r$ and $s$ commute;

\item[L4.]
\label{puresx}
\begin{enumerate}
\item  if $s\leq \, ^{\{s,t\}}(^Jw)$ and  $t\leq \, ^{\{s,t\}}(^Jw)$, then $M_{st}= \lambda_s$,
\item   if $s\leq \, ^{\{s,t\}}(^Jw)$ and  $t\not \leq \, ^{\{s,t\}}(^Jw)$, then $M_{st}$ commutes with $\rho_s$,
\item   if $s\not\leq \, ^{\{s,t\}}(^Jw)$ and  $t \leq \, ^{\{s,t\}}(^Jw)$, then $M_{st}$ commutes with $\rho_t$;
\end{enumerate}

\item[L5.] if $s\leq \,  (_{J}w)^{\{s\}}$, then $M_{st}$ commutes with $\lambda_s$ on $[e,w_0(s,t)]$.

\end{enumerate} 
\end{defi}
(As shown in \cite[Lemma~4.3]{Mchara}, Properties R5 and L5 are equivalent to the, a priori, more restrictive Properties R5 and L5 appearing in \cite{Mtrans}.)

Given a right system $\mathcal R=(J,s,t,M_{st})$ for $w$, the \emph{matching associated} with it is the map  $M_{\mathcal R}$ sending  $u\in [e,w]$ to 
 $$M_\mathcal R (u) = (u^{J})^{\{s,t\}}\, \cdot \, M_{st} \Big( (u^{J})_{\{s,t\}} \,  \cdot  \,  _{\{s\}} (u_{J}) \Big) \, 
\cdot \, ^{\{s\}}(u_{J}).$$
Symmetrically, the \emph{matching associated} with a left system $\mathcal L$ for $w$ is the map $_\mathcal L M$ sending 
$u\in [e, w]$ to 
\[
 _{\mathcal L}M(u)=
 (_Ju)^{\{s\}} \, \cdot \, M_{st} \Big( \, (_Ju)_{\{s\}}  \,  \cdot  \,  _{\{s,t\}} (^J  u) \Big) \, 
\cdot \, ^{\{s,t\}}(^J u),
\]
i.e., $_{\mathcal L}M(u)=\big(M_\mathcal L(u^{-1})\big)^{-1}$, where $M_\mathcal L$ is the map on $[e,w^{-1}]$ associated to $\mathcal L$ as a right system for $w^{-1}$.

The fact that $M_\mathcal R$ and $_{\mathcal L}M$ are actually  matchings of $w$ and the fact that the lengths add in these products  are shown in \cite{Mchara} (respectively, in Corollary~4.10 and Proposition~4.9). 

Note that $M_{\mathcal R}$ acts as $\lambda_s$ on $[e, w_0(s,r)]$ for all $r\in J$, and as $\rho_s$ on $[e, w_0(s,r)]$ for all $r\in S\setminus (J\cup \{t\})$; symmetrically, $_{\mathcal L}M$ acts as $\rho_s$ on $[e, w_0(s,r)]$ for all $r\in J$, and as $\lambda_s$ on $[e, w_0(s,r)]$ for all $r\in S\setminus (J\cup \{t\})$.

We comment  that, if $s\in D_R(w)$, $t\in S\setminus\{s\}$,  $J= \{s\}$ and $M_{st}= \rho_s$, then we obtain a right system whose associated  matching is the right multiplication matching $\rho_s$ ($M = \rho_s$ on the entire interval $[e,w]$). Symmetrically,  we obtain left multiplication matchings as special cases of matchings associated with  left systems.
On the other hands, we may obtain matchings that are not multiplication matchings. For example,  let $W$ be the Coxeter group of type $A_3$ with Coxeter generators $s_1$, $s_2$ and $s_3$ numbered as usual (i.e. $m_{s_1,s_2}=m_{s_2,s_3}=3$ and $m_{s_1,s_3}=2$), and let $w=s_1s_2s_3s_1 \in W$. The quadruple  $\mathcal R=(\{s_2,s_3\},s_2,s_1,M)$, with $M(e)=s_2$, $M(s_1)=s_1s_2$, and $M(s_2s_1)=s_1s_2s_1$, is a right system for $w$ whose associated matching is not a multiplication matching (the reader may check that the resulting matching is the dashed special matchings in the first picture of Figure~\ref{accoppiamento}).

The main result of  \cite{Mchara} is that the matchings arising from systems of $w$ are exactly the special matchings of $w$. We only need  one side of this characterization  (see \cite[Theorem~4.12]{Mchara}).
\begin{thm}
\label{caratteri}
Every  special matching of $w\in W$  is  associated with a right or a left system of $w$.  
\end{thm}

We refer the reader to \cite{Mchara2} for a more compact  characterization in terms of only one self-dual type of systems. 

\subsection{Kazhdan--Lusztig polynomials}
Given a Coxeter system $(W,S)$ and $H \subseteq S$,  the Bruhat order induces an ordering on the set of minimal coset representatives  $W^{H}$ and the parabolic  intervals $[u,v]^{H}$, for all $u,v \in W^{H}$. 

We introduce the  parabolic Kazhdan--Lusztig $R$-polynomials and the parabolic Kazhdan--Lusztig polynomials through the following theorems-definitions, which are due to Deodhar (see \cite[\S \S 2-3]{Deo87} for their proofs).
\begin{thm}
\label{7.1}
Let $(W,S)$ be a Coxeter system, and $H \subseteq S$.
 For each $x \in \{ -1,q \}$,
  there is a unique family of polynomials $\{ R_{u,v}^{H,x}(q)
\} _{u,v \in W^{H}} \subseteq {\bf Z}[q]$ such that, for all $u,v \in W^{H}$:
\begin{enumerate}
\item $R^{H,x}_{u,v}(q)=0$ if $u \not \leq v$;
\item $R^{H,x}_{u,u}(q)=1$;
\item if $u<v$ and $s \in D_{L}(v)$, then
\[ R_{u,v}^{H,x} (q)= \left\{ \begin{array}{ll}
R_{su,sv}^{H,x}(q), & \mbox{if $s \in D_{L}(u)$,} \\
(q-1)R_{u,sv}^{H,x}(q)+qR_{su,sv}^{H,x}(q), & \mbox{if $s \notin D_{L}(u)$
and $su \in W^{H}$,} \\
(q-1-x)R_{u,sv}^{H,x}(q), & \mbox{if $s \notin D_{L}(u)$ and $su \notin
W^{H}$.}
\end{array} \right. \]
\end{enumerate}
\end{thm}
In what follows, we often use the inductive formula of Theorem~\ref{7.1} without explicit mention.
\begin{thm}
\label{7.2}
Let $(W,S)$ be a Coxeter system, and $H \subseteq S$.
For each $x \in \{ -1,q \}$,
 there is a unique family of polynomials $\{ P^{H,x}_{u,v}(q) \} 
_{u,v \in  W^{H}}  \subseteq {\bf Z} [q]$, such that, for all 
$u,v \in W^{H}$:
\begin{enumerate}
\item $P^{H,x}_{u,v}(q)=0$ if $u \not \leq v$;
\item $P^{H,x}_{u,u}(q)=1$;
\item deg$(P^{H,x}_{u,v}(q)) \leq  \frac{1}{2}\left(
\ell(v)-\ell(u)-1 \right) $, if $u < v$;
\item 
\( q^{\ell(v)-\ell(u)} \, P^{H,x}_{u,v} \left( \frac{1}{q} \right)
= \sum _{z\in[u,v]_H}   R^{H,x}_{u,z}(q) \,
P^{H,x}_{z,v}(q). \)
\end{enumerate}
\end{thm}

The polynomials $R^{H,x}_{u,v}(q)$ and $P^{H,x}_{u,v}(q)$ 
are  the {\em
parabolic Kazhdan--Lusztig $R$-polynomials} and {\em parabolic  Kazhdan--Lusztig polynomials} of $W^{H}$ of type $x$. 
\begin{rem}
\label{senso}
For a fixed $H\subset S$, the parabolic Kazhdan--Lusztig $R$-polynomials and the parabolic Kazhdan--Lusztig polynomials are equivalent. More precisely, given $w\in W^H$, it is possible to compute the family $\{P_{u,v}^{H,x}(q)\}_{u,v\in [e,w]^H}$ once one knows   the family $\{R_{u,v}^{H,x}(q)\}_{u,v\in [e,w]^H}$, and vice versa.
\end{rem}

For $H=\emptyset$, $R_{u,v}^{\emptyset ,-1}(q)=R_{u,v}^{\emptyset ,q}(q)$ and $
P_{u,v}^{\emptyset ,-1}(q)=P_{u,v}^{\emptyset ,q}(q)$
 are the ordinary Kazhdan--Lusztig $R$-polynomials $R_{u,v}(q)$ and Kazhdan--Lusztig polynomials $P_{u,v}(q)$  of $W$.
 
The following result gives another relationship between the parabolic Kazhdan--Lusztig polynomials and their ordinary counterparts  (see \cite[Proposition 3.4, and Remark 3.8]{Deo87}).
\begin{pro}
\label{7.3}
Let $(W,S)$ be a Coxeter system, $H \subseteq S$, and $u,v \in W^{H}$.
We have 
\[ P_{u,v}^{H,q}(q)=\sum _{w \in W_{H}}(-1)^{\ell(w)}P_{uw,v}(q) .\]
Furthermore, if $W_{H}$ is finite, then
\[ P_{u,v}^{H,-1}(q)=P_{uw_{0}^{H},vw_{0}^{H}}(q) , \]
where $w_{0}^{H}$ is the longest element of $W_H$. 
\end{pro}

\section{Preliminary results}
In this section, we give some preliminary results that are needed to prove the main result of this work. 

For convenience, we state the  following straightforward result here for later reference.  
\begin{lem}
\label{ricorsioni}
The sequence  $\{R_i\}_{i\geq 1}\subseteq \mathbb Z[q]$, defined as
$$R_i=\begin{cases}
(q-1)  \Big( \sum_{k=0}^{i-1} (-1)^k q^k \Big), & \text{if $i$ is odd},\\
 (q-1)^2  \Big( \sum_{k=0}^{\frac{i-2}{2}}  q^{2k} \Big), & \text{if $i$ is even},
\end{cases}$$
is the unique sequence satisfying 
$$R_i= (q-1) R_{i-1} + q R_{i-2}  \qquad R_1=(q-1) \quad R_2=(q-1)^2$$
\end{lem}
\begin{proof}
Omitted.
\end{proof}
We observe the following fact. Let $W$ be a dihedral Coxeter groups, and $u,w\in W $. If   $\ell(w)-\ell(u)=i$, with $i \geq 1$, then  the ordinary Kazhdan--Lusztig $R$-polynomial $R_{u,v}(q)$ is the polynomial $R_i$ defined in Lemma \ref{ricorsioni}. In particular, as it is well-known, Conjecture~\ref{comb-inv-con} holds true for dihedral Coxeter groups since two intervals $[u,w]$ and $[u',w']$ in two dihedral Coxeter groups are isomorphic as posets if and only if $\ell(w)-\ell(u)=\ell(w')-\ell(u')$.

\bigskip

We fix an arbitrary Coxeter system $(W,S)$, a subset  $H \subset S$, and $s,t\in S$. For notational convenience, from now on we let $\bar{s}=t$ and $\bar{t}=s$.
Recall that, for every $x\in W$, the coset $W_{\{s,t\}} \;  x= \{g_{st} \, x :  g_{st} \in W_{\{s, t\}} \}$ is isomorphic, as a poset, to the dihedral Coxeter group $W_{\{s,t\}}$. 
\begin{pro}
\label{partizione}
Consider an arbitrary coset $W_{\{s,t\}} \cdot x$, where (we suppose without lack of generality) $ x \in \, ^{\{s, t\}}W$. The intersection   $(W_{\{s,t\}} \cdot x) \cap   W^H$ is one of the following set:
\begin{enumerate}
\item  $\emptyset $,
\item  $\{x\}$,
\item  $\{g_{st} \cdot x \,  : \,  g_{st}\in W_{\{s,t\}}, \, t\notin  D_R(g_{st})\}$,  
\item  $\{g_{st} \cdot x \, : \, g_{st}\in W_{\{s,t\}}, \, s\notin  D_R(g_{st})\}$, 
\item  $W_{\{s,t\}} \cdot x $.
\end{enumerate}

\end{pro}
\begin{proof}
First of all, recall that if an element $w$ belongs to $W^H$,  then $rw$ belongs to $W^H$  for all $r\in D_L(w)$. Also recall that an element not in $W^H$ have at most one coatom in $W^H$ (see \cite[Lemma 4.1]{Mtrans}); in particular, in the case $W_{\{s, t\}}$ is finite, the intersection $(W_{\{s,t\}} \cdot x) \cap   W^H$ cannot be $(W_{\{s,t\}} \cdot x) \setminus \{w_0 \cdot x\}$, where $w_0$ is the longest element of  $W_{\{s, t\}}$.

 We prove the statement by contradiction and, by what we have just recalled, we suppose  that there exist $g \in W_{\{s,t\}}\setminus \{e\}$ and $p\in \{s,t\}\setminus D_L(g)$ such that 
 
 \begin{itemize}
\item $g \cdot x \in W^H$,
\item  $p\cdot g\cdot x \notin W^H$,
\item  $p \cdot g $ is not the longest element of $ W_{\{s, t\}}$ (if any, i.e. in the case $W_{\{s, t\}}$ is finite).
\end{itemize}
 
 Since $p\cdot g \cdot x \notin W^H$, there exists $h\in H\cap D_R(p\cdot g \cdot x)$. Since $g \cdot x \in W^H$, we have $h\notin D_R(g \cdot x)$; by the Lifting Property (Lemma~\ref{ll}), $p\cdot g \cdot x= g\cdot x \cdot h$. Hence, both $s$ and $t$ belong to $D_L(p\cdot g\cdot x)$ and thus to $D_L(p\cdot g)$; by a well-known fact, this means that $p\cdot g$  is the longest element of $ W_{\{s, t\}}$.
\end{proof}
The following three results give formulas expressing some parabolic Kazhdan--Lusztig $R$-polynomials as linear combinations of other parabolic Kazhdan--Lusztig $R$-polynomials. 
   (The choice of the indices   could seem unnatural at this point: the reason for this choice is that, in Section~\ref{main}, we apply these results in a situation where we have  two missing parts $w_1$ and $u_1$, i.e.  two parts $w_1$ and $u_1$ that are both equal to $e$).
   \begin{lem}
\label{sesingleton}
Let $w= w_2 \cdot w_3 \in W^H$ and  $u= u_2 \cdot u_3 \in W^H$ with:
\begin{itemize}
\item $u\leq w$,
\item $ w_2,u_2 \in W_{\{s, t\}}$,  
\item $w_3, u_3 \in \,^{\{s, t\}}W$.
\end{itemize}
If $(W_{ \{s, t\} } \cdot u_3) \cap W^H= \{u_3 \}$, then  $u=u_3$ and 
$$ R^{H,x}_{u,w} (q)= (q-1-x) ^{\ell(w_2) }    R^{H,x}_{u_3,w_3} (q).$$
\end{lem}
\begin{proof}
If $w_2=e$, the assertions are trivial. Suppose $w_2\neq e$ and fix $r\in \{s,t\}\cap D_L(w_2)$.
By the recursive formula of Theorem~\ref{7.1} (with $r$ as left descent of $w$), we have $R^{H,x}_{u,w} (q)=(q-1-x)R^{H,x}_{u,rw} (q)$. We get the assertion by iteration.
\end{proof}

Recall that, for $r\in \{s,t\}$, we denote by $\bar{r}$ the element in $\{s,t\}\setminus \{r\}$.
\begin{lem}
\label{sepiena}
Let $w= w_2 \cdot w_3 \in W^H$ and  $u= u_2 \cdot u_3 \in W^H$ with:
\begin{itemize}
\item $u\leq w$,
\item $ w_2,u_2 \in W_{\{s, t\}}$,  
\item $w_3, u_3 \in \,^{\{s, t\}}W$,
\end{itemize}
and suppose $W_{ \{s, t\} } \cdot u_3 \subseteq W^H$.
Then
 there exists a set of polynomials $\{ p_{g_{st }} (q)\}_{g_{st}\in W_{\{s,t\}}}\subseteq \mathbb Z[q]$  such that  
\begin{eqnarray*}
R^{H,x}_{u,w} (q)&= &\sum_{ g_{st} \in W_{\{s, t\}} }  p_{g_{st }} (q) R^{H,x}_{g_{st} \cdot u_3,w_3} (q) \\
 R_{u,w} (q) &=& \sum_{ g_{st} \in W_{\{s, t\}} }  p_{g_{st }} (q) R_{g_{st}\cdot u_3,w_3} (q)
 \end{eqnarray*}
(in other words,  both the parabolic and the ordinary Kazhdan--Lusztig $R$-polynomials  indexed by $u$ and $w$ can be expressed as a linear combination of, respectively, the parabolic and the ordinary Kazhdan--Lusztig $R$-polynomials  indexed by $g_{st}\cdot u_3$ and $w_3$, with $g_{st}\in W_{\{s,t\}}$, and the two expressions have the same coefficients). 

If, moreover,  $|\{x\in \{s,t\} :x\leq w_3\}|\leq 1$, 
then $\ell(w_2)-\ell (u_2) \geq -1$ and the following statements hold.
\begin{itemize}
\item[D$_{-1}$.] 
If $\ell(w_2)-\ell (u_2) = -1 $, then
$$ R^{H,x}_{u,w} (q)= 
 R^{H,x}_{pu_3,w_3} (q) .$$
\item[D$_{0}$.] If $\ell(w_2)-\ell (u_2) = 0 $,  then 
$$ R^{H,x}_{u,w} (q)= 
\left\{
\begin{array}{ll}
 R^{H,x}_{u_3,w_3} (q)  &   \text{if $u_2=w_2$}\\
(q-1) R^{H,x}_{pu_3,w_3} (q)  &   \text{if $u_2 \neq w_2$.}
\end{array}
\right.
$$
\item[D$_{1}$.] If $\ell(w_2)-\ell (u_2) = 1 $,  then
$$ R^{H,x}_{u,w} (q)= 
\left\{
\begin{array}{ll}
(q-1)  R^{H,x}_{u_3,w_3} (q)+ q R^{H,x}_{pu_3,w_3} (q) ,  &     \text{if $w_2=u_2 \cdot p$}\\
(q-1)  R^{H,x}_{u_3,w_3} (q),  &   \text{otherwise.}
\end{array}
\right.
$$
\item[D$_{i}$.] If $\ell(w_2)-\ell (u_2) =i \geq 2$,  then
$$ R^{H,x}_{u,w} (q)= 
R_i  \cdot  R^{H,x}_{u_3,w_3} (q)+ q R_{i-1} \cdot R^{H,x}_{pu_3,w_3} (q) $$
where
the family of polynomials $\{R_j\}_{j\geq 1}$ is as in Lemma~\ref{ricorsioni}.
\end{itemize}
In the previous statements, if  $|\{x\in \{s,t\} :x\leq w_3\}|\leq 1$ then  $\{p\}= \{x\in \{s,t\} :x\leq w_3\}$, if $|\{x\in \{s,t\} :x\leq w_3\}|=0$ then $ R^{H,x}_{pu_3,w_3} (q)=0$.
\end{lem}
\begin{proof}
Let us prove the first statement. If $w_2=e$, it is  trivial. Suppose $w_2\neq e$ and fix $r\in \{s,t\}\cap D_L(w_2)$. We apply the recursive formula of Theorem~\ref{7.1} (with $r$ as a left descent of $w$) to compute both $R^{H,x}_{u,w} (q)$ and $ R_{u,w} (q)$. Since $W_{ \{s, t\} } \cdot u_3 \subseteq W^H$, we cannot fall into the case when the factor $(q-1-x)$ occurs, and the two computations agree. We get the assertion by iterating this argument.

\medskip 
Let us prove the second part of the lemma and so suppose $\{p\} \supseteq \{x\in \{s,t\} :x\leq w_3\}$.

In this proof, we use the Subword Property (Theorem~\ref{subword}), Property (\ref{0}), and the recursive formula of Theorem~\ref{7.1} several times without explicit mention; when we apply the  recursive formula of Theorem~\ref{7.1}, we never fall into the case the factor $(q-1-x)$ occurs, since $W_{ \{s, t\} } \cdot u_3 \subseteq W^H$.

Since  $|\{x\in \{s,t\}:x\leq w_3\}|\leq 1$, the longest subword of type $s\textrm{-}t\textrm{-}s\textrm{-}t\textrm{-}\cdots$ or $t\textrm{-}s\textrm{-}t\textrm{-}s\textrm{-}\cdots$ of any reduced expression for $w$  has length  at most $ \ell(w_2) +1$, and hence  $\ell(w_2)-\ell (u_2) \geq -1$.

\bigskip 

Proof of D$_{-1}$. \quad Since $\ell(w_2)-\ell (u_2) = -1 $, necessarily $u_2= w_2 \cdot p$, as otherwise $u$ could not be smaller than or equal to $w$. We have 
$$ R^{H,x}_{u,w} (q)= R^{H,x}_{w_2pu_3,w_2w_3} (q)=
 R^{H,x}_{pu_3,w_3} (q) .$$

\medskip
Proof of D$_{0}$. \quad If  $u_2=w_2$, the assertion is immediate. If  $u_2\neq w_2$, there exists an element $v\in W_{\{s,t\}}$, with $\ell(v)= \ell(w_2)-1=\ell(u_2)-1$, such that  
$u_2=v\cdot p$ and $w_2=l\cdot v$, where $l\in\{s,t\}\setminus D_L(v)$.
We have 
$$ R^{H,x}_{u,w} (q)= R^{H,x}_{vpu_3,lvw_3} (q)=(q-1) R^{H,x}_{vpu_3,vw_3} (q)+ q R^{H,x}_{lvpu_3,vw_3} (q)=(q-1) R^{H,x}_{pu_3,w_3} (q)+ q R^{H,x}_{lvpu_3,vw_3} (q).
$$
We cannot have $l\cdot v \cdot p \cdot u_3 \leq v \cdot w_3$, since  all subwords of any reduced expression of $ v \cdot w_3$ of type $s\textrm{-}t\textrm{-}s\textrm{-}t\textrm{-}\cdots$ or $t\textrm{-}s\textrm{-}t\textrm{-}s\textrm{-}\cdots$ have length at most $ \ell(v)+1$, while $\ell(l\cdot v \cdot p)=\ell(v)+2$. Hence $R^{H,x}_{lvpu_3,vw_3} (q)=0$,  as desired.

\medskip
Proof of D$_{1}$.  \quad Since $\ell(w_2)-\ell (u_2) = 1 $, we have either 
\begin{enumerate}
\item $w_2=u_2 \cdot r$,  where $r\in\{s,t\}\setminus D_R(u_2)$, or
\item $w_2= l \cdot u_2$, where $l\in\{s,t\}\setminus D_L(u_2)$, $u_2\neq e$, and $l \cdot u_2$ is not the longest element of $W_{\{s,t\}}$ (if any).
\end{enumerate}
In the first case, we have
$$ R^{H,x}_{u,w} (q)=  R^{H,x}_{u_2u_3,u_2 r w_3} (q)= R^{H,x}_{u_3, r w_3} (q)=
(q-1)   R^{H,x}_{u_3,w_3} (q) +q  R^{H,x}_{r u_3,w_3} (q).$$
If $r=p$, we get the assertion. If $r\neq p$, then $r\not\leq w_3$; thus $r u_3 \not \leq w_3$, and we get the assertion as well. 

In the second case, we have
\begin{eqnarray*}
  R^{H,x}_{u,w} (q) &=&  R^{H,x}_{u_2u_3, l u_2  w_3} (q)=(q-1)R^{H,x}_{u_2u_3, u_2  w_3} (q)+ q R^{H,x}_{l u_2u_3, u_2  w_3} (q)=(q-1)R^{H,x}_{u_3,  w_3} (q)
\end{eqnarray*}
since $l \cdot u_2 \cdot u_3 \not\leq  u_2 \cdot w_3$. 

\medskip
Proof of D$_{i}$. \quad
Suppose $\ell(w_2)-\ell (u_2) = 2 $; we have either 
\begin{itemize}
\item $w_2=u_2 \cdot r  \cdot \bar{r} $, where $r\in\{s,t\}\setminus D_R(u_2)$, or 
\item $w_2= l  \cdot u_2  \cdot r$, where $l\in\{s,t\}\setminus D_L(u_2)$, $r\in\{s,t\}\setminus D_R(u_2)$, $u_2\neq e$, and $l \cdot u_2 \cdot r$ is not the longest element of $W_{\{s,t\}}$ (if any).  
\end{itemize}
In the first case, we have
\begin{eqnarray*}
 R^{H,x}_{u,w} (q)&=& R^{H,x}_{u_2u_3,u_2 r \bar{r}  w_3} (q)=R^{H,x}_{u_3, r \bar{r}  w_3} (q)= (q-1)R^{H,x}_{u_3,  \bar{r} w_3} (q)+ q R^{H,x}_{r u_3,  \bar{r} w_3} (q)\\
&=&(q-1) [(q-1)   R^{H,x}_{u_3,w_3} (q)+ q R^{H,x}_{\bar{r}u_3,w_3} (q) ]+q[(q-1)  R^{H,x}_{r u_3,w_3} (q)+ q R^{H,x}_{\bar{r} r u_3,w_3} (q) ]\\
&=&(q-1)^2   R^{H,x}_{u_3,w_3} (q)+ q(q-1) [R^{H,x}_{\bar{r}u_3,w_3} (q)+  R^{H,x}_{r u_3,w_3} (q)].
\end{eqnarray*}
since $\bar{r} r u_3 \not \leq w_3$. Thus the assertion follows since $\{\bar{r}u_3, ru_3\}\cap \{x: x\leq w_3\}=
 \{pu_3\}\cap \{x:x\leq w_3\}$.
 
In the second case, we have
\begin{eqnarray*}
 R^{H,x}_{u,w} (q)&=& R^{H,x}_{u_2u_3, lu_2 r w_3} (q)=(q-1)R^{H,x}_{u_2u_3, u_2 r w_3} (q)+q R^{H,x}_{lu_2u_3, u_2 r w_3} (q)\\
&=& (q-1)R^{H,x}_{u_3, r w_3} (q)+q R^{H,x}_{lu_2u_3, u_2 r w_3} (q)\\
& = &(q-1)[(q-1)R^{H,x}_{u_3,  w_3} (q)+q R^{H,x}_{ru_3, w_3} (q)]+ q R^{H,x}_{lu_2u_3, u_2 r w_3} (q).
\end{eqnarray*}
If $r=p$, then $lu_2u_3 \not\leq  u_2 r w_3$: thus $R^{H,x}_{lu_2u_3, u_2 r w_3} (q)=0$ and we get the assertion.
If $r\neq p$, then $r u_3 \not \leq   w_3$ and thus $R^{H,x}_{ru_3,  w_3} (q)=0$; on the other hand, $lu_2u_3 \leq  u_2 r w_3$ and
$R^{H,x}_{lu_2u_3, u_2 r w_3} (q)= (q-1)R^{H,x}_{p u_3,  w_3} (q) $ by Statement $D_0$.

\medskip
Suppose $\ell(w_2)-\ell (u_2) = 3 $; we have either 
\begin{itemize}
\item $w_2=u_2 \cdot r \cdot  \bar{r}  \cdot r$, , where $r\in\{s,t\}\setminus D_R(u_2)$, or 
\item $w_2= l  \cdot u_2 \cdot r  \cdot \bar{r} $, where $l\in\{s,t\}\setminus D_L(u_2)$, $r\in\{s,t\}\setminus D_R(u_2)$, $u_2\neq e$, and $l \cdot u_2 \cdot r \cdot \bar{r} $ is not the longest element of $W_{\{s,t\}}$ (if any). 
\end{itemize}
In the first case, we have
\begin{eqnarray*}
 R^{H,x}_{u,w} (q)&=& R^{H,x}_{u_2u_3,u_2 r \bar{r} r w_3} (q)=R^{H,x}_{u_3, r \bar{r} r w_3} (q)= (q-1)R^{H,x}_{u_3,  \bar{r} r w_3} (q)+ q R^{H,x}_{r u_3,  \bar{r} r w_3} (q)\\
&=& (q-1)[(q-1)^2R^{H,x}_{u_3,  w_3} (q)+q(q-1) R^{H,x}_{pu_3, w_3} (q)] + q(q-1)R^{H,x}_{u_3,  w_3} (q)\\
 &=&R_3 \cdot R^{H,x}_{u_3,  w_3} (q)+ q R_2 \cdot  R^{H,x}_{pu_3, w_3} (q),
\end{eqnarray*}
by the assertion (already proved) for when the difference of the length is equal to 2, and by Statement $D_1$.

In the second case, we have
\begin{eqnarray*}
 R^{H,x}_{u,w} (q)&=& R^{H,x}_{u_2u_3, lu_2 r \bar{r} w_3} (q)=(q-1)R^{H,x}_{u_2u_3, u_2 r \bar{r}  w_3} (q)+q R^{H,x}_{lu_2u_3, u_2 r \bar{r}  w_3} (q)\\
&=& (q-1)[(q-1)^2R^{H,x}_{u_3,  w_3} (q)+q(q-1) R^{H,x}_{pu_3, w_3} (q)] +q (q-1)R^{H,x}_{u_3,  w_3} (q) \\
&=&R_3 \cdot R^{H,x}_{u_3,  w_3} (q)+ q R_2 \cdot  R^{H,x}_{pu_3, w_3} (q)
\end{eqnarray*}
by the assertion (already proved) for when the difference of the length is equal to 2, and by Statement $D_1$.

\bigskip

Suppose $\ell(w_2)-\ell (u_2)=i $, with $i\geq 4$, and use induction on $\ell(w_2)$.
The base of the induction is $u_2=e$ and $w_2\in \{p \bar{p} p \bar{p},  \bar{p}p \bar{p}p \}$:  the assertion follows by a direct computation that we omit.

Let $\ell(w_2)>4$ and  $h \in D_L(w_2)$. If $h \in D_L(u_2)$, then $ R^{H,x}_{u,w} (q)= R^{H,x}_{u_2u_3,w_2w_3} (q)=R^{H,x}_{hu_2u_3,hw_2w_3} (q)$ and we may conclude by the  induction hypothesis since $\ell(hw_2)< \ell(w_2)$, and   $\ell(w_2) - \ell(u_2)=\ell(hw_2) - \ell(hu_2)$. 
If $h\not \in D_L(u_2)$,  then 
\begin{eqnarray*}
 R^{H,x}_{u,w} (q)&=&R^{H,x}_{u_2u_3, w_2 w_3} (q)= (q-1)R^{H,x}_{u_2u_3, h w_2  w_3} (q)+q R^{H,x}_{ hu_2u_3, hw_2   w_3} (q)\\
&=& (q-1)[R_{i-1} \cdot R^{H,x}_{u_3,  w_3} (q)+ qR_{i-2} \cdot  R^{H,x}_{pu_3, w_3} (q)] +
q [R_{i-2} \cdot R^{H,x}_{u_3,  w_3} (q)+ q R_{i-3} \cdot  R^{H,x}_{pu_3, w_3} (q)]\\
&=& [(q-1)R_{i-1} + q R_{i-2}] \cdot R^{H,x}_{u_3,  w_3} (q)+ q[(q-1) R_{i-2} + q R_{i-3} \cdot  R^{H,x}_{pu_3, w_3} (q)] \\
&=&R_i \cdot R^{H,x}_{u_3,  w_3} (q)+ qR_{i-1} \cdot  R^{H,x}_{pu_3, w_3} (q)
\end{eqnarray*}
where the last equation follows by Lemma~\ref{ricorsioni}.
\end{proof}

\begin{rem}
Lemma~\ref{sesingleton} and the first part of Lemma~\ref{sepiena} hold more generally (with the same straightforward proof) if we replace $\{s,t\}$ with an arbitrary subset $J\subseteq S$. 
\end{rem}

In the proof of the following result, as well as in the proof of the main result of this work, it is essential   $x\in \{q, -1\}$; indeed, we repeatedly use that $x$ satisfies
\begin{eqnarray}
\label{1}
(q-1) (q-1-x) +q=(q-1-x)^2.
\end{eqnarray}

\begin{lem}
\label{protutti}
Let $w= w_2 \cdot w_3 \in W^H$ and  $u= u_2 \cdot u_3 \in W^H$ with:
\begin{itemize}
\item $u\leq w$,
\item $w_2, u_2 \in W_{\{s,t\}}$, 
\item $w_3, u_3 \in \,^{\{s,t\}}W$,
\item  $|\{x\in \{s,t\} :x\leq w_3\}|\leq 1$.
\end{itemize}
Suppose that  $(W_{ \{s, t\} } \cdot u_3) \cap W^H$ is a chain (see Proposition~\ref{partizione}) and let $r,\bar{r}\in \{s,t\}$ be such that $r \cdot u_3\in W^H$ and $\bar{r} \cdot u_3\notin W^H$.

Then $\ell(w_2)-\ell (u_2) \geq -1$ and the following statements hold.
\begin{itemize}
\item[D$_{-1}$.] 
If $\ell(w_2)-\ell (u_2) = -1 $, then
$$ R^{H,x}_{u,w} (q)= 
 R^{H,x}_{ru_3,w_3} (q) .$$
\item[D$_{0}$.] If $\ell(w_2)-\ell (u_2) = 0 $,  then 
$$ R^{H,x}_{u,w} (q)= 
\left\{
\begin{array}{ll}
 R^{H,x}_{u_3,w_3} (q)  &   \text{if $u_2=w_2$}\\
(q-1) R^{H,x}_{ru_3,w_3} (q)  &   \text{if $u_2 \neq w_2$}
\end{array}
\right.
$$
\item[D$_{1}$.] If $\ell(w_2)-\ell (u_2) = 1 $,  then
$$ R^{H,x}_{u,w} (q)= 
\left\{
\begin{array}{ll}
(q-1-x)   R^{H,x}_{u_3,w_3} (q)  &   \text{if $\bar{r}\in D_R(w_2)$}\\
(q-1)  R^{H,x}_{u_3,w_3} (q)  &   \text{if $\bar{r}\notin D_R(w_2)$ and $u_2\neq e$}\\
(q-1) R^{H,x}_{u_3,w_3} (q)+ q R^{H,x}_{ru_3,w_3} (q) ,  &     \text{if $\bar{r}\notin D_R(w_2)$ and $u_2= e$}
\end{array}
\right.
$$
\item[D$_{2}$.] If $\ell(w_2)-\ell (u_2) = 2 $,  then
$$ R^{H,x}_{u,w} (q)= 
\left\{
\begin{array}{ll}
(q-1-x) [(q-1)   R^{H,x}_{u_3,w_3} (q)+ q R^{H,x}_{ru_3,w_3} (q) ],  &   \text{if $r\in D_R(w_2)$}\\
(q-1) [(q-1-x)   R^{H,x}_{u_3,w_3} (q)+ q R^{H,x}_{ru_3,w_3} (q) ],  &   \text{if $r\notin D_R(w_2)$}
\end{array}
\right.
$$
\item[D$_{i}$.] If $\ell(w_2)-\ell (u_2) \geq 3$,  then
$$ R^{H,x}_{u,w} (q)= 
(q-1) (q-1-x) ^{\ell(w_2) - \ell(u_2)-2} [(q-1-x)  R^{H,x}_{u_3,w_3} (q)+ q R^{H,x}_{ru_3,w_3} (q)  ]
$$
\end{itemize}
\end{lem}
\begin{proof}
In this proof, we use the Subword Property (Theorem~\ref{subword}), Property (\ref{0}), and the recursive formula of Theorem~\ref{7.1} several times without explicit mention.   

Note that we have $D_R(u_2)=\{r\}$ unless $u_2=e$; in particular, $u_2$ cannot be  the top element of $W_{\{s,t\}}$ (if any). If $u_2\neq e$, we let  $l\in {\{s,t\}}$ 
 be such that $\{l\}=D_L(u_2)$, so that  $u_2$ has a (unique) reduced expression starting with  $ l$ and ending with $r$. 

Since  $|\{x\in \{s,t\}:x\leq w_3\}|\leq 1$, the longest subword of type $s\textrm{-}t\textrm{-}s\textrm{-}t\textrm{-}\cdots$ or $t\textrm{-}s\textrm{-}t\textrm{-}s\textrm{-}\cdots$ of any reduced expression for $w$  has length  at most $ \ell(w_2) +1$, and hence  $\ell(w_2)-\ell (u_2) \geq -1$.

\bigskip 

Proof of D$_{-1}$. \quad Since $\ell(w_2)-\ell (u_2) = -1 $, necessarily $u_2= w_2 \cdot r$, as otherwise $u$ could not be smaller than or equal to $w$. We have 
$$ R^{H,x}_{u,w} (q)= R^{H,x}_{w_2ru_3,w_2w_3} (q)=
 R^{H,x}_{ru_3,w_3} (q) .$$

\medskip
Proof of D$_{0}$. \quad If  $u_2=w_2$, the assertion is immediate. If $u_2\neq w_2$, there exists an element $v\in W_{\{s,t\}}$, with $\ell(v)= \ell(w_2)-1=\ell(u_2)-1$, such that 
$u_2=v\cdot r$ and $w_2=\bar{l}\cdot v$.
We have 
$$ R^{H,x}_{u,w} (q)= R^{H,x}_{vru_3,\bar{l}vw_3} (q)=(q-1) R^{H,x}_{vru_3,vw_3} (q)+ q R^{H,x}_{\bar{l}vru_3,vw_3} (q)=(q-1) R^{H,x}_{ru_3,w_3} (q)+ q R^{H,x}_{\bar{l}vru_3,vw_3} (q).
$$
We cannot have  $\bar{l}\cdot v \cdot r \cdot u_3 \leq v \cdot w_3$, since  all subwords of any reduced expression of $ v \cdot w_3$ of type $s\textrm{-}t\textrm{-}s\textrm{-}t\textrm{-}\cdots$ or $t\textrm{-}s\textrm{-}t\textrm{-}s\textrm{-}\cdots$ have length at most $\ell(v)+1$ while $\ell(\bar{l}\cdot v \cdot r)=\ell(v)+2$. Hence $R^{H,x}_{\bar{l}vru_3,vw_3} (q)=0$, as desired.

\medskip
Proof of D$_{1}$.  \quad Since $\ell(w_2)-\ell (u_2) = 1 $, we have either 
\begin{enumerate}
\item $w_2=u_2 \cdot \bar{r}$, or 
\item $w_2= \bar{l} \cdot u_2 \neq u_2 \cdot \bar{r}$ and $u_2\neq e$, or 
\item $w_2=r$ and $u_2=e$. 
\end{enumerate}
In the first case, we have
$$ R^{H,x}_{u,w} (q)=  R^{H,x}_{u_2u_3,u_2 \bar{r} w_3} (q)= R^{H,x}_{u_3, \bar{r} w_3} (q)=
(q-1-x)   R^{H,x}_{u_3,w_3} (q) .$$
In the second case, we have
\begin{eqnarray*}
  R^{H,x}_{u,w} (q) &=&  R^{H,x}_{u_2u_3,\bar{l} u_2  w_3} (q)=(q-1)R^{H,x}_{u_2u_3, u_2  w_3} (q)+ q R^{H,x}_{\bar{l} u_2u_3, u_2  w_3} (q)=(q-1)R^{H,x}_{u_3,  w_3} (q)
\end{eqnarray*}
since $\bar{l} \cdot u_2 \cdot u_3 \not\leq  u_2 \cdot w_3$. In the third case, the assertion is immediate.

\medskip
Proof of D$_{2}$. \quad
Since $\ell(w_2)-\ell (u_2) = 2 $, we have either 
\begin{itemize}
\item $w_2=u_2 \cdot \bar{r}  \cdot r $, or 
\item $w_2= \bar{l}  \cdot u_2  \cdot \bar{r}\neq  u_2  \cdot \bar{r}  \cdot r$ (where we set $\bar{l}=r$ if $u_2=e$). 
\end{itemize}
In the first case, we have
\begin{eqnarray*}
 R^{H,x}_{u,w} (q)&=& R^{H,x}_{u_2u_3,u_2 \bar{r} r w_3} (q)=R^{H,x}_{u_3, \bar{r} r w_3} (q)= (q-1-x)R^{H,x}_{u_3,  r w_3} (q)\\
&=& (q-1-x) [(q-1)   R^{H,x}_{u_3,w_3} (q)+ q R^{H,x}_{ru_3,w_3} (q) ].
\end{eqnarray*}
In the second case, we have
\begin{eqnarray*}
 R^{H,x}_{u,w} (q)&=& R^{H,x}_{u_2u_3, \bar{l}u_2 \bar{r} w_3} (q)=(q-1)R^{H,x}_{u_2u_3, u_2 \bar{r} w_3} (q)+q R^{H,x}_{\bar{l}u_2u_3, u_2 \bar{r} w_3} (q)\\
&=& (q-1)(q-1-x)R^{H,x}_{u_3,  w_3} (q)+q(q-1) R^{H,x}_{ru_3, w_3} (q),
\end{eqnarray*}
where the last equality follows from Statements D$_{1}$ and D$_{0}$.

\medskip
Proof of D$_{i}$. \quad
If $\ell(w_2)-\ell (u_2) = 3 $, we have either 
\begin{itemize}
\item $w_2=u_2 \cdot  \bar{r}  \cdot r  \cdot \bar{r}$, or 
\item $w_2= \bar{l}  \cdot u_2  \cdot \bar{r}  \cdot r\neq u_2  \cdot \bar{r} \cdot  r  \cdot \bar{r}$ (where we set $\bar{l}=r$ if $u_2=e$).
\end{itemize}
In the first case, we have
\begin{eqnarray*}
 R^{H,x}_{u,w} (q)&=& R^{H,x}_{u_2u_3,u_2 \bar{r} r \bar{r}w_3} (q)=R^{H,x}_{u_3, \bar{r} r \bar{r}w_3} (q)= (q-1-x)R^{H,x}_{u_3,  r\bar{r} w_3} (q)\\
&=&(q-1-x) (q-1) [(q-1-x)   R^{H,x}_{u_3, w_3} (q)+ q R^{H,x}_{ru_3, w_3} (q) ],
\end{eqnarray*}
by Statement D$_{2}$.
In the second case, we have
\begin{eqnarray*}
 R^{H,x}_{u,w} (q)&=& R^{H,x}_{u_2u_3, \bar{l}u_2 \bar{r} pw_3} (q)=(q-1)R^{H,x}_{u_2u_3, u_2 \bar{r} r w_3} (q)+q R^{H,x}_{\bar{l}u_2u_3, u_2 \bar{r} r w_3} (q)\\
&=& (q-1)(q-1-x)[(q-1)R^{H,x}_{u_3,  w_3} (q)+qR^{H,x}_{ru_3,  w_3} (q)]+
q(q-1)R^{H,x}_{u_3,  w_3} (q)\\
&=& (q-1)[(q-1)(q-1-x)+q]R^{H,x}_{u_3,  w_3} (q)+ q(q-1)(q-1-x)R^{H,x}_{ru_3,  w_3} (q)\\
&=&(q-1)(q-1-x)^2R^{H,x}_{u_3,  w_3} (q)+ q(q-1)(q-1-x)R^{H,x}_{ru_3,  w_3} (q)
\end{eqnarray*}
by Statements D$_{2}$ and D$_{1}$, and by Eq. (\ref{1}).

If $\ell(w_2)-\ell (u_2) = 4 $, we have either
 \begin{itemize}
\item $w_2=u_2  \cdot \bar{r}  \cdot r  \cdot \bar{r}  \cdot r$, or 
\item $w_2= \bar{l} \cdot  u_2  \cdot \bar{r}  \cdot r \cdot  \bar{r}\neq u_2  \cdot \bar{r}  \cdot r  \cdot \bar{r} \cdot  r$  (where we set $\bar{l}=r$ if $u_2=e$).
\end{itemize}
In the first case, we have
\begin{eqnarray*}
 R^{H,x}_{u,w} (q)&=& R^{H,x}_{u_2u_3,u_2 \bar{r} r \bar{r}rw_3} (q)=R^{H,x}_{u_3, \bar{r} r \bar{r} rw_3} (q)\\
&=& (q-1-x)R^{H,x}_{u_3,  r\bar{r}r w_3} (q)=(q-1)(q-1-x)^2 [ (q-1-x)   R^{H,x}_{u_3, w_3} (q)+ \\
&&  q R^{H,x}_{ru_3, w_3} (q) ].
\end{eqnarray*} 
In the second case, we have
\begin{eqnarray*}
 R^{H,x}_{u,w} (q)&=& R^{H,x}_{u_2u_3, \bar{l}u_2 \bar{r} r \bar{r}w_3} (q)=(q-1)R^{H,x}_{u_2u_3, u_2 \bar{r} r \bar{r}w_3} (q)+q R^{H,x}_{\bar{l}u_2u_3, u_2 \bar{r} r \bar{r}w_3} (q)\\
&=& (q-1)^2 (q-1-x)[ (q-1-x)   R^{H,x}_{u_3, w_3} (q)+ q R^{H,x}_{ru_3, w_3} (q)]+ \\
&&q(q-1) [(q-1-x)   R^{H,x}_{u_3, w_3} (q)+ q R^{H,x}_{ru_3, w_3} (q)]\\
&=&(q-1)(q-1-x)[(q-1) (q-1-x)+q]  R^{H,x}_{u_3, w_3} (q)+ \\
&&q (q-1)[ (q-1)(q-1-x)+q]R^{H,x}_{ru_3, w_3} (q)\\
&=&(q-1)(q-1-x)^3  R^{H,x}_{u_3, w_3} (q)+ q (q-1)(q-1-x)^2R^{H,x}_{ru_3, w_3} (q)
\end{eqnarray*}
where the last equality follows by Eq. (\ref{1}). 
In both cases we have used  statements that we have already proved.

Suppose $\ell(w_2)-\ell (u_2) \geq 5$ and use induction on $\ell(w_2)$.
The base of the induction is $u_2=e$ and $w_2\in \{r \bar{r} r \bar{r}r,  \bar{r}r \bar{r}r \bar{r}\}$:  the assertion follows by a direct computation that we omit.

Let $\ell(w_2)>5$: if $l\in D_L(w_2)$, then $ R^{H,x}_{u,w} (q)= R^{H,x}_{u_2u_3,w_2w_3} (q)=R^{H,x}_{lu_2u_3,lw_2w_3} (q)$ and we may conclude by the induction hypothesis since  $\ell(lw_2)< \ell(w_2)$ and   $\ell(w_2) - \ell(u_2)=\ell(lw_2) - \ell(lu_2)$. 
If $l\not \in D_L(w_2)$,  then consider $\bar{l}\in D_L(w_2)$:  we have  $\bar{l} \notin D_L(u)$ since $\bar{l}\notin D_L(u_2)$, and $\bar{l}u\in W^H$  since $u_2\neq e$. Hence, using the induction hypothesis, we have
\begin{eqnarray*}
 R^{H,x}_{u,w} (q)&=& (q-1)R^{H,x}_{u_2u_3, \bar{l}w_2 w_3} (q) + q R^{H,x}_{\bar{l} u_2u_3, \bar{l}w_2 w_3} (q)\\
&=&(q-1)^2 (q-1-x) ^{\ell(\bar{l}w_2) - \ell(u_2)-2} 
[(q-1-x)   R^{H,x}_{u_3,w_3} (q)+ q    R^{H,x}_{ru_3,w_3} (q)  ]  \\
& &
+
q  (q-1) (q-1-x) ^{\ell(\bar{l}w_2) - \ell(\bar{l}u_2)-2}
[ (q-1-x)   R^{H,x}_{u_3,w_3} (q)+ q R^{H,x}_{ru_3,w_3} (q)  ]\\
&=&
(q-1) (q-1-x) ^{\ell(w_2) - \ell(u_2)-3}[ (q-1) (q-1-x) +q]
  R^{H,x}_{u_3,w_3} (q)
+\\
&&
q (q-1) (q-1-x) ^{\ell(w_2) - \ell(u_2)-4}[(q-1) (q-1-x) +q] 
 R^{H,x}_{ru_3,w_3} (q) \\
&=&
(q-1) (q-1-x) ^{\ell(w_2) - \ell(u_2)-2}
[ (q-1-x)
  R^{H,x}_{u_3,w_3} (q)
+
q R^{H,x}_{ru_3,w_3} (q) ]
\end{eqnarray*}
where the last equality follows by Eq. (\ref{1}). 
\end{proof}

\begin{rem}
\label{sololunghezza}
Lemma~\ref{protutti} implies that, under its hypotheses, $R^{H,x}_{u,w}$ is a combination of $R^{H,x}_{u_3,w_3}$  and $R^{H,x}_{pu_3,w_3}$ with coefficients in $\mathbb Z[q]$. Furthermore, if either
\begin{itemize}
\item $\ell(w_2)-\ell(u_2) \geq 3$, or
\item $\ell(w_2)-\ell(u_2) = 2$ and  $\{x\in \{ s,t\} : x\leq w_3\}=\emptyset$,
\end{itemize}
then the  coefficients of the combination depend only on 
$\ell(w_2)-\ell(u_2) $.
\end{rem}

\section{Main result}
\label{main}
In this section, we prove Theorem~\ref{computa}, whose implications were discussed in Section~1. In particular, for any arbitrary Coxeter system $(W,S)$,  any arbitrary subset $H\subseteq S$, and any arbitrary element  $w\in W^H$, we give an algorithm  for computing the parabolic Kazhan--Lusztig $R$-polynomials $\{R_{u,w}^{H,x}(q)\}_{u\in W^H}$ once one knows the poset-isomorphism class  of the interval $[e,w]$, and which  elements of the interval $[e,w]$ belong to $W^H$.

As an immediate corollary, we have that it is possible  to compute also the parabolic Kazhdan--Lusztig polynomials $\{P_{u,w}^{H,x}(q)\}_{u\in W^H}$ from the knowledge only of the poset-isomorphism class of the interval $[e,w]$ and which  elements of the interval $[e,w]$ belong to $W^H$ (see Remark~\ref{senso}).

First, we give the following general definition.
\begin{defi}
Let $P$ be a poset and $T\subseteq P$ be a subposet of $P$. A \emph{relative special matching of $P$ with respect to $T$} is a special matching $M$ of $P$ such that, if $p\in T$ and $M(p)\lhd p$, then $M(p) \in T$.  
\end{defi}

Now, fix    an arbitrary Coxeter system $(W,S)$, a subset $H \subseteq S$, and an element $w \in W^H$. An {\em $H$-special matching of $w$} is a relative special matching of $[e,w]$ with respect to $[e,w]^H$, that is a special matching $M$ of $w$ such that, if $u \leq w$, $u \in W^H$, and $M(u) \lhd u $, then $M(u) \in W^H.$

Note that the $\emptyset$-special matchings are exactly the special matchings and that all left multiplication matchings  are $H$-special,
for all $H \subseteq S$.

We say that an $H$-special matching $M$ of $w$ \emph{calculates} the parabolic Kazhdan--Lusztig $R$-polynomials (or, simply,  is \emph{calculating}) provided 
\begin{equation}
 R_{u,w}^{H,x} (q)= \left\{ \begin{array}{ll}
R_{M(u),M(w)}^{H,x}(q), & \mbox{if $M(u)  \lhd u$,} \\
(q-1)R_{u,M(w)}^{H,x}(q)+qR_{M(u),M(w)}^{H,x}(q), & \mbox{if $M(u) \rhd u$
and $M(u) \in W^{H}$,} \\
(q-1-x)R_{u,M(w)}^{H,x}(q), & \mbox{if $M(u) \rhd u$ and  $M(u) \notin W^{H}$.} 
\end{array} \right. 
\end{equation}
 for all $u \in W^H$ with $u \leq w$.
Clearly, all left multiplication matchings are calculating. 
Actually, our target is to prove that all $H$-special matchings are calculating.

We need the following result (see \cite[Theorem 4.2]{Mtrans}).
\begin{thm}
\label{orbite}
Let $M$ be an $H$-special matching of $w$. 
If
\begin{itemize}
\item every $H$-special matching of $v$ is calculating, for all  \( v \in W ^H\) with $v < w$, and
\item there exists a calculating special matching $N$ of $w$  commuting with $M$ and such that $M(w) \neq N(w)$,
\end{itemize}
then $M$  is calculating.
\end{thm}

We also need the following easy lemma.
\begin{lem}
\label{nicola}
Let $s,t\in S$, $g_{st}\in W_{\{s,t\}}$, $p\in  D_R(g_{st})$, and $M$ be a special matching of the dihedral interval $[e,g_{st}]$. If
\begin{itemize}
\item $M$ commutes with $\rho_p$, and
\item $M(x)\neq \rho_p(x)$, for all $x\in W_{\{s,t\}}$ such that $\ell(x)\neq 0,1$ and, if $W_{\{s,t\}}$ is finite, $\ell(x)\neq m_{s,t}-1, m_{s,t}$,
\end{itemize}
then $M$ is a left multiplication matching.
\end{lem}
\begin{proof}
Without loss of generality, we suppose  $M(e)=s$. 
We need to show $M(x)=sx$, for all $x\in [e,g_{st}]$. 
By contradiction, let $x$ be minimal such that $M(x)\neq sx$. 

Clearly $x \notin \{ e,s\}$. By minimality,  $x \lhd M(x)$, and $s \notin D_L(x)$ as otherwise $M(sx)$ would be $x$ since $sx$ would be smaller than $x$. 
Moreover, if $W_{\{s,t\}}$ is finite and $g_{st}$ is its  longest element $w_0$, then $x \notin \{w_0,sw_0\}$.  
The element $M(x)$ cannot be $xp$ (by hypothesis, since at least one among $x$ and $M(x)$ has length not in $\{0,1, m_{s,t}-1, m_{s,t}\}$); $M(x)$ cannot be  $sx$ (by assumption); $M(x)$ cannot be $tx$ (since $tx\lhd x$). The only possibility left is  $M(x)=x\bar{p}$ (recall that  $\bar{p}$ is the element in $\{s,t\}\setminus \{p\}$ and notice that, if $W_{\{s,t\}}$ is finite,  $x\bar{p}$ is not $ w_0$  since otherwise $w_0$ would be equal to $sx$). Hence the element $\rho_p M \rho_p (xp)$, which is $x\bar{p}p$, would have length equal to $\ell(xp)+3$, and so $M \rho_p (xp) \neq  \rho_p M (xp)$, which contradicts the fact that $M$ commutes with $\rho_p$.
\end{proof}

We now recall and prove Theorem~\ref{computa}.
\begin{thm*}
\label{tutti}
Given an arbitrary Coxeter system $(W,S)$ and a subset  $H \subset S$, let $w$ be any element in $W^H$.
 Then all $H$-special matchings of $w$ calculate the parabolic Kazhdan--Lusztig $R$-polynomials of $W^H$.
\end{thm*}
\begin{proof}
We use induction on $\ell(w)$, the case $\ell(w)\leq 1$ being trivial. Suppose $\ell(w)> 1$.

 Let $M$ be an $H$-special matching of $w$ and $u \in W^H$, with $u \leq w$. We need to show 
\begin{equation*}
\label{ecco}
 R_{u,w}^{H,x} (q)= \left\{ \begin{array}{ll}
R_{M(u),M(w)}^{H,x}(q), & \mbox{if $M(u)  \lhd u$,} \\
(q-1)R_{u,M(w)}^{H,x}(q)+qR_{M(u),M(w)}^{H,x}(q), & \mbox{if $M(u) \rhd u$
and $M(u) \in W^{H}$,} \\
(q-1-x)R_{u,M(w)}^{H,x}(q), & \mbox{if $M(u) \rhd u$ and  $M(u) \notin W^{H}$.} 
\end{array} \right. 
\end{equation*}

We may suppose that $M$ 
does not agree with a left multiplication matchings on both $u$ and $w$, because otherwise the assertion is clear since left multiplication matchings are calculating.

If there exists a left multiplication matching $\lambda$ of $w$ commuting with $M$ such that $\lambda(w) \neq M(w)$, then we can conclude by Theorem~\ref{orbite}. 

By Theorem~\ref{caratteri}, $M$ is associated with a system  $(J,s,t,M_{st})$. 
Suppose first that $(J,s,t,M_{st})$ is a right system and $(w^{J})^{\{s,t\}} \neq e$. Fix $l\in D_L((w^{J})^{\{s,t\}})$; thus $l\in D_L(w)$ and $\lambda_l$ is a special matching of $w$ that satisfies  $M(w) \neq \lambda_l(w)$ since 
 $$M(w) =(w^{J})^{\{s,t\}}\, \cdot \, M_{st} \Big ((w^{J})_{\{s,t\}} \, \cdot \, _{\{s\}} (w_{J})\Big )\, \cdot \,
 ^{\{s\}}(w_{J}) $$
while
$$\lambda_l(w) = l  (w^{J})^{\{s,t\}}\, \cdot \,  (w^{J})_{\{s,t\}} \, \cdot \, _{\{s\}} (w_{J})\, \cdot \,
 ^{\{s\}}(w_{J}) .$$
We need to show that $M$ and $\lambda_l$ commute. In order to apply Lemma~\ref{commutano}, we distinguish the following cases.\\
(a) $l\notin \{s,t\}$.\\
By Property R3 of the definition of a right system, either $l \notin J$ or $l$ commutes with $s$. In the first case, $M$ acts as $\rho_s$ on  $[e,w_0(s,l)]$  and hence commutes with $\lambda_l$ on  $[e,w_0(s,l)]$. In the second case, $M$ and $\lambda_l$ clearly commutes on  $[e,w_0(s,l)]$ since  $[e,w_0(s,l)]$  is a dihedral interval with  4 elements.\\
(b) $l=t$.\\
By Property R4, $M$ commutes with $\lambda_t$ on  $[e,w_0(s,t)]$.\\
(c) $l=s$.\\
We need to show that $M$ and $\lambda_s$ commute on every lower dihedral intervals $[e,w_0(s,r)]$, with $r\in S\setminus \{s\}$. For $r=t$, it follows from  Property R4. For $r\neq t$, $M$ acts on $[e,w_0(s,r)]$ as $\rho_s$ or $\lambda_s$, and in both cases $M$ commutes with $\lambda_s$ on $[e,w_0(s,r)]$.

\bigskip

Now suppose  that $(J,s,t,M_{st})$ is a left system and  $(_{J}w)^{\{s\}} \neq e$.
Fix $l\in D_L((_{J}w)^{\{s\}})$; thus $l\in D_L(w)$ and $\lambda_l$ is a special matching of $w$ that satisfies  $M(w) \neq \lambda_l(w)$ since 
 $$M(w) = (_Jw)^{\{s\}} \, \cdot \, M_{st} \Big( \, (_Jw)_{\{s\}}  \,  \cdot  \,  _{\{s,t\}} (^J  w) \Big) \, 
\cdot \, ^{\{s,t\}}(^J w) $$
while
$$\lambda_l(w) = l   (_Jw)^{\{s\}} \, \cdot  \, (_Jw)_{\{s\}}  \,  \cdot  \,  _{\{s,t\}} (^J  w)  \, 
\cdot \, ^{\{s,t\}}(^J w).$$
In order to  show that $M$ and $\lambda_l$ commute, we again apply Lemma~\ref{commutano}. If $l\neq s$, then $M$ acts as $\rho_s$ on  $[e,w_0(s,l)]$  and hence commutes with $\lambda_l$. Suppose
  $l=s$; we need to show that $M$ and $\lambda_s$ commute on every lower dihedral intervals $[e,w_0(s,r)]$, with $r\in S\setminus \{s\}$. If $r=t$, it follows from  Property L5. If $r\neq t$, then $M$ acts on $[e,w_0(s,r)]$ as $\rho_s$ or $\lambda_s$, and so $M$ commutes with $\lambda_s$.

\bigskip

Hence we may suppose that either
\begin{enumerate}
\item $(J,s,t,M_{st})$ is a right system and $(w^{J})^{\{s,t\}} = e$, or
\item $(J,s,t,M_{st})$ is a left system and $(_{J}w)^{\{s\}} = e$.
\end{enumerate}
In the first case, we set $w_2=  (w^{J})_{\{s,t\}} \, \cdot \, _{\{s\}} (w_{J})$, $w_3= \,
 ^{\{s\}}(w_{J})$,  $u_2=  (u^{J})_{\{s,t\}} \, \cdot \, _{\{s\}} (u_{J})$, $u_3= \,
 ^{\{s\}}(u_{J})$. In the second case, we set $w_2=  (_{J}w)_{\{s\}} \, \cdot \, _{\{s,t\}} (^{J}w)$, $w_3= \,^{\{s,t\}}(^{J}w)$,  
 $u_2=  (_{J}u)_{\{s\}} \, \cdot \, _{\{s,t\}} (^{J}u)$, $u_3= \,^{\{s,t\}}(^{J}u)$. 
In both cases, we get the $W_{\{s,t\}}\times  \, ^{\{s,t\}}W$-factorization of $w$ and $u$: 
$$w= w_2 \cdot w_3 \quad \quad \quad u= u_2 \cdot u_3.$$
Note:
\begin{itemize}
\item $u\leq w$,
\item $w_2, u_2 \in W_{\{s,t\}}$,
\item $w_3, u_3 \in \,^{\{s,t\}}W$,
\item $u_3\leq w_3$,
\item $ |\{x\in \{s,t\}: x\leq w_3\}|\leq 1$ (in the first case, this is trivial since $t\notin J$ and $w_3\in W_J$; in the second case $s$ and $t$ cannot be both smaller than or equal to $w_3$ since otherwise $M_{st}=\lambda_s$  by Property L4 and hence $M=\lambda_s$) and, if this cardinality is 1, we let $p$ be such that  $\{p\}=\{x\in \{s,t\} :x\leq w_3\}$,
\item $M(w)=M(w_2)\cdot w_3$ and $M(u)=M(u_2)\cdot u_3$,
\item $M$ acts as $\lambda_s$ on $[e,w_0(s,r)]$, for all $r\in S\setminus \{t\}$ such that $r \leq w$, 
\item if  $\{p\}=\{x\in \{s,t\} :x\leq w_3\}$, then $M$ commutes with $\rho_p$ on $[e,w_0(s,t)]$ (by either Property~R5 or Property~L4).
\end{itemize}

Recall Proposition~\ref{partizione}. If $(W_{ \{s, t\} } \cdot u_3) \cap W^H=\{u_3\}$, then we may conclude  using Lemma~\ref{sesingleton}. 

We now suppose  $W_{ \{s, t\} } \cdot u_3 \subseteq W^H$ and  apply Lemma~\ref{sepiena}.
Suppose  $\ell(w_2)-\ell(u_2)= i $, $i\geq 2$. 
 By Lemma~\ref{sepiena}, we have
\begin{eqnarray*} 
R^{H,x}_{u,w} (q)  &= & R_i \cdot  R^{H,x}_{u_3,w_3} (q)  + qR_{i-1}  \cdot   R^{H,x}_{pu_3,w_3} (q).
\end{eqnarray*}

If $M(u)\lhd u$, then $\ell(M(w_2)) - \ell(M(u_2))=i$ and also $R^{H,x}_{M(u),M(w)} (q)$ is equal to  $ R_i \cdot  R^{H,x}_{u_3,w_3} (q)  + qR_{i-1}  \cdot   R^{H,x}_{pu_3,w_3}$, by  Lemma~\ref{sepiena}.

Suppose $u \lhd M(u) $. 
If  $\ell(w_2)-\ell(u_2) \geq 4$ then 
$$(q-1)  R^{H,x}_{u_2u_3,M(w_2)w_3} (q)+  q   R^{H,x}_{M(u_2)u_3,M(w_2)w_3} (q)$$ is equal to
\begin{eqnarray*}
 &=& (q-1)[R_{i-1} \cdot  R^{H,x}_{u_3,w_3} (q)  + qR_{i-2}  \cdot   R^{H,x}_{pu_3,w_3}] +
 q[R_{i-2} \cdot  R^{H,x}_{u_3,w_3} (q)  + qR_{i-3}  \cdot   R^{H,x}_{pu_3,w_3}] \\
 &=&[ (q-1)R_{i-1} + qR_{i-2}] \cdot  R^{H,x}_{u_3,w_3} (q) + q[ (q-1)R_{i-2} + qR_{i-3}] \cdot   R^{H,x}_{pu_3,w_3}\\
&=&  R_i \cdot  R^{H,x}_{u_3,w_3} (q)  + qR_{i-1}  \cdot   R^{H,x}_{pu_3,w_3}
\end{eqnarray*}
 by Lemmas~\ref{ricorsioni} and~\ref{sepiena}, as desired.

If  $\ell(w_2)-\ell(u_2) = 3$, then Lemma~\ref{sepiena} implies  that 
$$(q-1)  R^{H,x}_{u_2u_3,M(w_2)w_3} (q)+  q   R^{H,x}_{M(u_2)u_3,M(w_2)w_3} (q)$$ is equal to
\begin{eqnarray}
\label{chicche}
 && (q-1)[R_{2} \cdot  R^{H,x}_{u_3,w_3} (q)  + qR_{1}  \cdot   R^{H,x}_{pu_3,w_3}] +
 q[(q-1)   R^{H,x}_{u_3,w_3} (q)  + q     R^{H,x}_{pu_3,w_3} \cdot  \chi]
\end{eqnarray}
where $\chi=\begin{cases}
1&\text{if $M(w_2)=M(u_2)\cdot p$}\\
0 & \text{otherwise.}
\end{cases}$

The term $ R^{H,x}_{pu_3,w_3} \cdot  \chi$ is always $0$: indeed, if $p\leq w_3$, then $M(w_2)\neq M(u_2)\cdot p$ since, otherwise, $w_2=M(M(u_2)\cdot p) = M \circ \rho_p \, (M(u_2))\neq \rho_p \circ  M \, (M(u_2))=u_2p$ and $M$ would not commute with $\rho_p$ on $[e,w_0(s,t)]$. Hence the polynomial in (\ref{chicche}) is always equal to 

\begin{eqnarray*}
  &=&[ (q-1)R_{2} + q(q-1)]   R^{H,x}_{u_3,w_3} (q) +  q(q-1)R_{1} \cdot   R^{H,x}_{pu_3,w_3}\\
&=&  R_3 \cdot  R^{H,x}_{u_3,w_3} (q)  + qR_2  \cdot   R^{H,x}_{pu_3,w_3},
\end{eqnarray*}
as desired.

Suppose   $\ell(w_2)-\ell(u_2) = 2$. If $\{x\in \{s,t\}: x\leq w_3\}= \emptyset$,  then 
$$(q-1) R^{H,x}_{u_2u_3,M(w_2)w_3} (q)+  q   R^{H,x}_{M(u_2)u_3,M(w_2)w_3} (q)$$ is equal to $(q-1) ^2 R^{H,x}_{u_3,w_3} (q)$ and the assertion follows. Suppose  $\{x\in \{s,t\}: x\leq w_3\}= \{p\}$ and recall that, in this case, $M$ commutes with $\rho_p$ on $[e,w_0(s,t)]$:
in order to compute 
\begin{eqnarray}
\label{roma}
&&(q-1)  R^{H,x}_{u_2u_3,M(w_2)w_3} (q)+  q   R^{H,x}_{M(u_2)u_3,M(w_2)w_3} (q)
\end{eqnarray}
we distinguish two cases, according to as whether $p$ is in $D_R(M(u_2))$ or not.
If $p\in  D_R(M(u_2))$, then there exist $r\in \{s,t\}$ and $v\in W_{\{s,t\}}$ with $\ell(v)=\ell(u_2)$  such that $M(u_2)=v \cdot p$ and $M(w_2)= r \cdot v$. 
The polynomial in (\ref{roma}) is equal to 
 \begin{eqnarray*}
&=&(q-1)  R^{H,x}_{u_2u_3,r v w_3} (q)+  q   R^{H,x}_{vpu_3, rv w_3} (q)\\
& = & (q-1)^2 R^{H,x}_{u_3, w_3} (q) + q(q-1) R^{H,x}_{pu_3,  w_3} (q)
\end{eqnarray*}
since  Lemma~\ref{sepiena} implies
\begin{itemize}
\item $ R^{H,x}_{u_2u_3,r v w_3} (q)=(q-1)R^{H,x}_{u_3, w_3} (q)$ (notice  $r\cdot v \neq u_2p $),
\item $R^{H,x}_{vpu_3, rv w_3} (q)=(q-1) R^{H,x}_{pu_3,  w_3} (q)$  (notice  $vpu_3 \leq  rv w_3$ if and only if $pu_3 \leq w_3$).
\end{itemize}

If $p\notin  D_R(M(u_2))$, then $M(u_2) \cdot p = M(u_2 \cdot p)= w_2$ and $M(w_2)= u_2 \cdot p$, as otherwise $\rho_p \circ M (u_2) \neq M \circ \rho_p (u_2)$. Thus $M(u_2)\not\leq M(w_2) > M(w_2) p$ and 
$M(u_2) \cdot u_3 \not \leq M(w_2) \cdot w_3$, and so the polynomial in (\ref{roma}) is equal to 
 \begin{eqnarray*}
&=&(q-1)  R^{H,x}_{u_2u_3, u_2 p w_3} (q)=  (q-1)[(q-1) R^{H,x}_{u_3, w_3} (q) + q R^{H,x}_{pu_3,  w_3} (q)].
\end{eqnarray*}

\bigskip 
Suppose now  $\ell(w_2)-\ell(u_2) = 1$. By Lemma~\ref{sepiena}, we have 
$$ R^{H,x}_{u,w} (q)= 
\left\{
\begin{array}{ll}
(q-1)  R^{H,x}_{u_3,w_3} (q)+ q R^{H,x}_{pu_3,w_3} (q) ,  &     \text{if $w_2=u_2 \cdot p$}\\
(q-1)  R^{H,x}_{u_3,w_3} (q),  &   \text{otherwise.}
\end{array}
\right.
$$
If $M(u_2)\rhd u_2$, then $M(u_2)=w_2$, since $M$ is a special matching, and 
$$(q-1)  R^{H,x}_{u_2u_3,M(w_2)w_3} (q)+  q   R^{H,x}_{M(u_2)u_3,M(w_2)w_3} (q)$$
is equal to
 \begin{eqnarray*}
&=&(q-1)  R^{H,x}_{u_2u_3,u_2w_3} (q)+  q   R^{H,x}_{w_2u_3,u_2w_3} (q)\\
&=&(q-1)  R^{H,x}_{u_3,w_3} (q)+  q   R^{H,x}_{w_2u_3,u_2w_3} (q)
\end{eqnarray*}
where $ R^{H,x}_{w_2u_3,u_2w_3} (q)= 
\left\{
\begin{array}{ll}
 R^{H,x}_{pu_3,w_3} (q) ,  &     \text{if $w_2=u_2 \cdot p$}\\
0,  &   \text{otherwise.}
\end{array}
\right.
$

If $M(u_2)\lhd u_2$, then either $\{x\in\{s,t\}: x\leq w_3\}=\emptyset$, or $\{x \in \{s,t\} : x \leq w_3\}=\{p\}$ and $w_2=u_2\cdot p$ if and only if $M(w_2)=M(u_2) \cdot p$ since $M$ and $\rho_p$ commute. Hence $ R^{H,x}_{u,w} (q)=R^{H,x}_{M(u),M(w)} (q)$, by Lemma~\ref{sepiena}.

Suppose   $\ell(w_2)-\ell(u_2) = 0$. If $w_2=u_2$ the assertion is trivial. Otherwise, $u_2 \lhd M(u_2)=w_0(s,t)$,  $M(w_2)\lhd w_2$ and $M(w_2) \lhd u_2$. Since $u\leq w$, necessarily $p\leq w_3$, $p\in D_R(u_2)$,  $u_2= (lw_2)\cdot p$ where $l\in D_L(w_2)\setminus D_L(u_2)$, $w_0(s,t)=w_2\cdot p= l \cdot u_2$ and, since $M\circ \rho_p ( w_0(s,t))=\rho_p\circ M ( w_0(s,t))$, we have  $u_2=M(w_2) \cdot p$. Thus
$$(q-1)  R^{H,x}_{u_2u_3,M(w_2)w_3} (q)+  q   R^{H,x}_{M(u_2)u_3,M(w_2)w_3} (q)$$
is equal to
 \begin{eqnarray*}
&=&(q-1)  R^{H,x}_{M(w_2) p u_3,M(w_2)w_3} (q)= (q-1)  R^{H,x}_{p u_3,w_3} (q)
\end{eqnarray*}
since $M(u_2)u_3 \not \leq M(w_2)w_3$, and we conclude by  Lemma~\ref{sepiena}.

Suppose   $\ell(w_2)-\ell(u_2) = -1$. Thus $u_2=w_2 \cdot p$ as otherwise $u\not \leq v$, and $u_2= w_0(s,t)$. Hence $M(u_2)\lhd u_2$ and we conclude by  Lemma~\ref{sepiena}.

\bigskip

We are left with the case  when $(W_{ \{s, t\} } \cdot u_3) \cap W^H$ is a chain and we apply Lemma~\ref{protutti}. Let $r,\bar{r}\in \{s,t\}$ be such that $ru_3\in W^H$ and $\bar{r}u_3\notin W^H$.

\medskip
Suppose   $\ell(w_2)-\ell(u_2) \geq 4$. 
 By Lemma~\ref{protutti}, we have
\begin{eqnarray}
\label{pomeriggio}
\qquad 
R^{H,x}_{u,w} (q)  &= &  (q-1) (q-1-x) ^{\ell(w_2) - \ell(u_2)-2}   [(q-1-x) R^{H,x}_{u_3,w_3} (q)  +  q    R^{H,x}_{ru_3,w_3} (q)].
\end{eqnarray}

If $M(u)\lhd u$, then we can conclude since $\ell(M(w_2)) - \ell(M(u_2))=\ell(w_2) - \ell (u_2)$  (see Remark~\ref{sololunghezza}).  

If $u \lhd M(u) \notin W^H$, then 
\begin{eqnarray*}
R_{u,M(w)}^{H,x}(q)&=& (q-1) (q-1-x) ^{\ell(M(w_2)) - \ell(u_2)-2}
 [(q-1-x)  R^{H,x}_{u_3,w_3} (q)+ q R^{H,x}_{ru_3,w_3} (q)  ]\\
&=& (q-1) (q-1-x) ^{\ell(w_2) - \ell(u_2)-3}   [(q-1-x) R^{H,x}_{u_3,w_3} (q)  +  q     R^{H,x}_{ru_3,w_3} (q)]
\end{eqnarray*}
 by Lemma~\ref{protutti}. Hence $(q-1-x)R_{u,M(w)}^{H,x}(q)$ is equal to the right side  of equation (\ref{pomeriggio}), as desired.

If $u \lhd M(u) \in W^H$, we separate two cases. 
If either $\ell(w_2)-\ell(u_2) \geq 5$, or $\ell(w_2)-\ell(u_2) = 4$ and  $\{x\in \{s,t\}: x\leq w_3\}=\emptyset$, then 
$$(q-1)  R^{H,x}_{u_2u_3,M(w_2)w_3} (q)+  q   R^{H,x}_{M(u_2)u_3,M(w_2)w_3} (q)$$ is equal to
\begin{eqnarray*}
 &=& (q-1)  (q-1) (q-1-x) ^{\ell(M(w_2)) - \ell(u_2)-2}  [(q-1-x)  R^{H,x}_{u_3,w_3} (q) + 
q    R^{H,x}_{ru_3,w_3} (q) ]+\\
& & q (q-1) (q-1-x) ^{\ell(M(w_2)) - \ell(M(u_2))-2}    [(q-1-x)R^{H,x}_{u_3,w_3} (q) + 
 q   R^{H,x}_{ru_3,w_3} (q) ] \\
&=&  (q-1) (q-1-x) ^{\ell(w_2) - \ell(u_2)-3}   [ (q-1) (q-1-x) +q]   R^{H,x}_{u_3,w_3} (q)  + \\
& & q(q-1) (q-1-x) ^{\ell(w_2) - \ell(u_2)-4}   [ (q-1) (q-1-x) +q]   R^{H,x}_{ru_3,w_3} (q)   \\
&=&  (q-1) (q-1-x) ^{\ell(w_2) - \ell(u_2)-2}    [(q-1-x)R^{H,x}_{u_3,w_3} (q)  + 
 q R^{H,x}_{ru_3,w_3} (q)]
\end{eqnarray*}
 by Lemma~\ref{protutti} and Eq. (\ref{1}), as desired.

If $\ell(w_2)-\ell(u_2) = 4$ and  $|\{x\in \{s,t\}: x\leq w_3\}|=1$, then, by Lemma~\ref{protutti}, 
$(q-1)  R^{H,x}_{u_2u_3,M(w_2)w_3} (q)+  q   R^{H,x}_{M(u_2)u_3,M(w_2)w_3} (q)$ is equal to
\begin{eqnarray*}
 &=& (q-1)  (q-1) (q-1-x)   [ (q-1-x)R^{H,x}_{u_3,w_3} (q) + q   R^{H,x}_{ru_3,w_3} (q) ]+\\
& & q[ (q-1) (q-1-x)    R^{H,x}_{u_3,w_3} (q) + q  A  R^{H,x}_{ru_3,w_3} (q)  ] \\
 &=& (q-1) (q-1-x) [ (q-1) (q-1-x) + q]  R^{H,x}_{u_3,w_3} (q) + \\
& & q[ (q-1)^2 (q-1-x) +  q A]    R^{H,x}_{ru_3,w_3} (q)\\
&=& (q-1) (q-1-x)^3   R^{H,x}_{u_3,w_3} (q) +  q[ (q-1)^2 (q-1-x) +  q A]   R^{H,x}_{ru_3,w_3} (q)   
\end{eqnarray*}
where $A= 
\left\{
\begin{array}{ll}
(q-1-x),  &   \text{if $r\in D_R(w_2)$}\\
(q-1),  &   \text{if $r\notin D_R(w_2)$}
\end{array}
\right. 
$ and the last equation holds by  Eq. (\ref{1}).
 If $r\cdot u_3 \not\leq w_3$, we are done.
Let us show that $r \cdot u_3$ cannot be  smaller than or equal to $ w_3$ by contradiction. We would have $r\leq w_3$ (so $r=p$) and $M$ would commute with $\rho_r$ on $[e,w_0(s,t)]$: on the other hand, $M(x)\neq xr$ for all $x\in W_{\{s,t\}}$ with $\ell(x)\neq 0,1,m_{s,t}-1, m_{s,t}$, since $M$ is $H$-special and $M(x\cdot u_3)$ cannot be $xr \cdot u_3$. By Lemma~\ref{nicola}, these two facts together would imply that $M$ is a left multiplication matching on $[e,w_0(s,t)]$, and $M$ would be a left multiplication matching on $[e,w]$, which is a contradiction.

\bigskip
Suppose   $\ell(w_2)-\ell(u_2) =3$.
 By Lemma~\ref{protutti}
\begin{eqnarray*}
 R^{H,x}_{u,w} (q)  &= &  (q-1) (q-1-x)    [(q-1-x) R^{H,x}_{u_3,w_3} (q) +  q R^{H,x}_{ru_3,w_3} (q)].
\end{eqnarray*}

If $M(u)\lhd u$, then we conclude by Lemma~\ref{protutti}.  

If $u \lhd M(u) \notin W^H$, then 
$$
R_{u,M(w)}^{H,x}(q)=  \left\{
\begin{array}{ll}
(q-1-x) [(q-1)   R^{H,x}_{u_3,w_3} (q)+ q R^{H,x}_{ru_3,w_3} (q) ],  &   \text{if $r\in D_R(w_2)$}\\
 (q-1) [(q-1-x)   R^{H,x}_{u_3,w_3} (q)+ q R^{H,x}_{ru_3,w_3} (q) ],  &   \text{if $r\notin D_R(w_2)$}
\end{array}
\right.
$$
by Lemma~\ref{protutti}. If $r \cdot u_3 \not\leq w_3$, we are done.

Let us show that $r \cdot u_3$ cannot be  smaller than or equal to $ w_3$ by contradiction. 

We would have $r=p$, and $M$ would commute with $\rho_r$ on $[e,w_0(s,t)]$ since $r \cdot u_3 \leq w_3$ implies $r\leq w_3$: on the other hand, $M(x)\neq xr$ for all $x\in W_{\{s,t\}}$ with $\ell(x)\neq 0,1,m_{s,t}-1, m_{s,t}$, since $M$ is $H$-special and $M(x\cdot u_3)$ cannot be $xr\cdot u_3$. By Lemma~\ref{nicola}, these two facts together would imply that $M$ is a left multiplication matching on $[e,w_0(s,t)]$
and so $M$ would be a left multiplication matching also on $[e,w]$, which is impossible.

If $u \lhd M(u) \in W^H$,  then
\begin{itemize}
\item $M(u_2)=l \cdot u_2$, where $l\in \{s,t\}\setminus D_{L}(u_2)$ and $l \cdot u_2$ is not the longest element in $W_{\{s,t\}}$ (if any),
\item $M(u_2\cdot \bar{r})= l\cdot u_2\cdot \bar{r}$ since $M$ is $H$-special (as otherwise we would have $M(u_2\cdot\bar{r}\cdot u_3)= \bar{l}\cdot l \cdot u_2 \cdot u_3\in W^H$, with $u_2\cdot \bar{r}\cdot u_3\notin W^H$).
\end{itemize}
Hence the only possibility is $M(w_2)= \bar{l}\cdot l \cdot u_2$ (recall  $M(w_2)\lhd w_2$) and
 $w_2\in\{ l \cdot \bar{l} \cdot l \cdot u_2,  \bar{l} \cdot l \cdot u_2 \cdot \bar{r}\}$. But $w_2= l \cdot \bar{l} \cdot l \cdot u_2$ is not allowed 
  since $M$ would agree with  $\lambda_l$  on both $w$ and $u$, which is impossible.
Thus  $w_2=   \bar{l} \cdot l \cdot u_2 \cdot \bar{r}\neq l \cdot \bar{l} \cdot l \cdot u_2$, and 
$$(q-1)  R^{H,x}_{u_2u_3,M(w_2)w_3} (q)+  q   R^{H,x}_{M(u_2)u_3,M(w_2)w_3} (q)$$ is equal to
\begin{eqnarray*}
   &=& (q-1)  R^{H,x}_{u_2u_3, \bar{l} l u_2w_3} (q)+  q   R^{H,x}_{l u_2u_3, \bar{l} l u_2w_3} (q) \\
   &=& (q-1)  (q-1-x)[(q-1) R^{H,x}_{u_3,w_3} (q) + q   R^{H,x}_{ru_3,w_3} (q) ]+ q(q-1) R^{H,x}_{u_3,w_3} (q)\\
& =& (q-1) [ (q-1) (q-1-x)+q]   R^{H,x}_{u_3,w_3} (q) + q  (q-1) (q-1-x) R^{H,x}_{ru_3,w_3} (q)\\
&= & (q-1) (q-1-x)^2   R^{H,x}_{u_3,w_3} (q) + q  (q-1) (q-1-x) R^{H,x}_{ru_3,w_3} (q)
\end{eqnarray*}
 by Lemma~\ref{protutti} and Eq. (\ref{1}), as desired.

\bigskip
Suppose   $\ell(w_2)-\ell(u_2) =2$.

If $M(u) \lhd u $,  then $M(u_2)=lu_2$ with $l\in D_L(u_2)$  (since $M$ is $H$-special),  $M(M(u_2)\cdot \bar{r})= l \cdot M(u_2) \cdot \bar{r}=u_2 \cdot \bar{r}$, $M(w_2)= M(u_2) \cdot \bar{r} \cdot r$, and $w_2$ cannot be $ l \cdot \bar{l} \cdot u_2 $ since otherwise $M$ and  $\lambda_l$ would agree on both $w$ and $u$. Thus  $w_2= \bar{l} \cdot u_2 \cdot \bar{r} \neq  l \cdot \bar{l} \cdot u_2$, and $w_2= M \circ \rho_r (M(u_2)\cdot \bar{r}) \neq \rho_r \circ M (M(u_2)\cdot \bar{r})$, which implies $r\not\leq w_3$: hence $R^{H,x}_{u,w} (q) =R^{H,x}_{M(u),M(w)} (q)$ since they are both equal to  
$$  (q-1) (q-1-x) R^{H,x}_{u_3,w_3} (q)$$
by Lemma~\ref{protutti}.

If $u \lhd M(u) \notin W^H$, then $M(u_2)= u_2 \cdot \bar{r}$, $w_2\in \{l\cdot u_2 \cdot \bar{r}, \bar{l} \cdot l \cdot u_2\}$, and $M(w_2)= l \cdot u_2$, where 
$l\in \{s,t\} \setminus  D_L(u_2)$ if $u_2\neq e$ and $l= r$ if $u_2=e$.
We have
\begin{eqnarray*}
R_{u,M(w)}^{H,x}(q) &=& R_{u_2u_3,lu_2w_3}^{H,x}(q)=  (q-1)   R^{H,x}_{u_2u_3,u_2w_3} (q)+ q R^{H,x}_{lu_2u_3,u_2w_3} (q) \\
&=& (q-1)   R^{H,x}_{u_3,w_3} (q)+ q R^{H,x}_{lu_2u_3,u_2w_3} (q) 
\end{eqnarray*}
where the last term is $0$ unless $u_2=e$ (and so $l=r$) and $r \cdot u_3\leq w_3$ (so $r=p$).
If $\{x\in\{s,t\}: x \leq w_3\}$ is either empty or  $\{\bar{r}\}$, we conclude by Lemma~\ref{protutti}. If  $\{x\in\{s,t\}: x \leq w_3\}=\{r\}$,  then $M$ must commute with $\rho_r$ on $[e, w_0(s,t)]$: this implies $u_2=e$ (since $\rho_r \circ M(u_2)= u_2\cdot \bar{r} \cdot r$ while  $ \rho_r (y)\lhd y$ for all $y\in (W^H\cap W_{\{s,t\}})\setminus \{e\}$), and 
  $w_2=\bar{r}\cdot  r\cdot u_2= \bar{r}\cdot  r$ (since $\bar{r}\cdot  r$ must be the element of length 2 in $W_{\{s,t\}}$ covering its matched element as otherwise $\rho_r \circ M(e)$ could not be equal to $ M \circ \rho_r(e)$). The assertion then follows.

If $u \lhd M(u) \in W^H$,  then $M(u_2)= l \cdot u_2$ with $l \in \{s,t\} \setminus  D_L(u_2)$ (and $l=r$ if $u_2=e$),  and $M(w_2)= u_2 \cdot \bar{r}$. Moreover, $w_2\neq l \cdot  u_2 \cdot  \bar{r}$ (as otherwise $M$ and  $\lambda_l$ would agree on both $w$ and $u$). Hence $ w_2$ should be equal to $  \bar{l}\cdot  l \cdot  u_2$ but also this is not possible since the element $w_2\cdot u_3=\bar{l} \cdot  l \cdot  u_2 \cdot  u_3 \in W^H$ (which belongs to $[e,w]$ since $u_3 \leq w_3$) would be matched with $u_2 \cdot \bar{r} \cdot u_3\notin W^H$, and this contradicts the definition of $H$-special.

\bigskip
Suppose   $\ell(w_2)-\ell(u_2) =1$.

If $M(u) \lhd u $,  then $M(u_2)=lu_2$ with $l\in D_L(u_2)$ (since $M$ is $H$-special), and $M(w_2)=(lu_2) \cdot \bar{r}$. Now $w_2 \in \{ \bar{l} \cdot u_2, u_2 \cdot \bar{r} \}$, but actually both possibilities are not permitted.  On one hand, $w_2$ cannot be  $u_2 \cdot \bar{r}$, since otherwise  $M$ and $\lambda_l$  would agree on both $u$ and $w$. On the other hand, $w_2 \neq \bar{l} \cdot u_2$ since otherwise $\bar{l} \cdot u_2  \cdot u_3 \rhd M(\bar{l} \cdot u_2  \cdot u_3)= M(\bar{l} \cdot u_2)  \cdot u_3=  lu_2  \cdot \bar{r}  \cdot u_3 \notin W^H$, with $\bar{l}u_2  \cdot u_3 \in W^H$, which is impossible since $M$ is $H$-special. 

If $u\lhd M(u)$, then $M(u_2)=w_2$ since $M$ is a special matching. The element $w_2$ cannot be  $l\cdot u_2$, with $l\notin D_L(u)$, since otherwise  $M$ and $\lambda_{l}$ would agree on both $u$ and $w$. Thus $w_2= u_2 \cdot \bar{r}$ and 
$$R^{H,x}_{u,w} (q)=  (q-1-x) R^{H,x}_{u_3,w_3} (q)$$
 by Lemma~\ref{protutti}: on the other hand, $M(u)=w_2 \cdot u_3 =u_2 \cdot \bar{r} \cdot u_3\notin W^H$ and 
$$ R^{H,x}_{u,M(w)} (q)=R^{H,x}_{u_2u_3,u_2 w_3} (q)= R^{H,x}_{u_3,w_3} (q),$$ and the assertion follows.

\bigskip
Suppose  $\ell(w_2)-\ell(u_2) =0$.

If $u_2=w_2$, then the result is clear. Otherwise $r\leq w_3$ and $M$ commutes with $\rho_r$: we have $(M(w_2),M(u_2))= (l w_2 ,  l u_2)$,  with $l\in D_L(w_2)\setminus D_L(u_2)$, and hence $M$ coincides with $\lambda_{l}$  on both $u$ and $w$, which is a contradiction.

\bigskip
Suppose   $\ell(w_2)-\ell(u_2) =-1$.

Necessarily $u_2=w_2\cdot r= w_0(s,t)$ and Lemma \ref{5.3} implies $M(u_2)\lhd u_2$. Clearly, we also have    $M(w_2)\lhd w_2$. Hence $R^{H,x}_{u,w} (q)$ and $R^{H,x}_{M(u),M(w)} (q)$ coincide, since they are both equal to $R^{H,x}_{ru_3,w_3} (q)$ by Lemma~\ref{protutti}. 

\bigskip
The proof is completed.
\end{proof}

We illustrate Theorems~\ref{elena}, \ref{congettura}, and \ref{computa} with an example. Let $W$ be the Coxeter group of type $A_3$ with Coxeter generators $s_1$, $s_2$ and $s_3$ numbered as usual (i.e. $m_{s_1,s_2)}=m_{s_2,s_3}=3$ and $m_{s_1,s_3}=2$). Let $H=\{s_2\}$, $w=s_1s_2s_3s_1 \in W^H$, and $u=s_1 \in W^H$. 

Suppose that we want to compute $R_{u,w}^{H,x}(q)$ but we only know the isomorphism class of the poset $[e,w]$ and which elements of $[e,w]$ belong to $W^H$ and which do not. In other words, we know the pieces of information that we can detect from Figure~\ref{figura}, where the elements represented by full (respectively, empty) bullets belong to (respectively, do not belong to) $W^H$.
\begin{figure}[h]
\begin{center}
$$
\begin{tikzpicture}

\draw[fill=black]{(0,0) circle(3pt)}; \node[below] at (0,-0.1){$e$};

\draw[fill=black]{(-1,1) circle(3pt)};
\draw[fill=black]{(1,1) circle(3pt)}; \node[right] at (1.1,1){$u$};
\draw[fill=white]{(0,1) circle(3pt)};

\draw[fill=black]{(-1.5,2) circle(3pt)};
\draw[fill=black]{(-0.5,2) circle(3pt)};
\draw[fill=black]{(0.5,2) circle(3pt)};
\draw[fill=white]{(1.5,2) circle(3pt)};

\draw[fill=black]{(-1,3) circle(3pt)};
\draw[fill=black]{(1,3) circle(3pt)};
\draw[fill=white]{(0,3) circle(3pt)};

\draw[fill=black]{(0,4) circle(3pt)};
\node[above] at (0,4.1){$w$};

 \draw{(0,0)--(0,0.9)}; 
 \draw{(-1,1)--(0,0)}; 
  \draw{(1,1)--(0,0)}; 

  \draw{(-1,1)--(-0.5,2)}; 
   \draw{(-1,1)--(0.5,2)}; 
    \draw{(-0.05,1.1)--(-0.5,2)}; 
     \draw{(0.11,1.07)--(1.42,1.95)}; 
      \draw{(-0.07,1.05)--(-1.5,2)};
\draw{(1,1)--(-1.5,2)}; 
\draw{(1,1)--(0.5,2)}; 
\draw{(1,1)--(1.45,1.9)}; 

\draw{(-1.5,2)--(-1,3)}; 
\draw{(-1.5,2)--(-0.08,2.92)}; 
\draw{(-0.5,2)--(-1,3)}; 
\draw{(-0.5,2)--(1,3)}; 
\draw{(0.5,2)--(-1,3)};
\draw{(0.5,2)--(1,3)}; 
\draw{(1.45,2.1)--(1,3)}; 
\draw{(1.415,2.05)--(0.1,2.95)}; 

\draw{(0,4)--(1,3)};
\draw{(0,4)--(0,3.1)};
\draw{(0,4)--(-1,3)};

\end{tikzpicture}$$
\end{center}
\caption{\label{figura} The Hasse diagram of $[e,w]$ and how $[e,w]^H$ embeds in $[e,w]$.}
\end{figure}
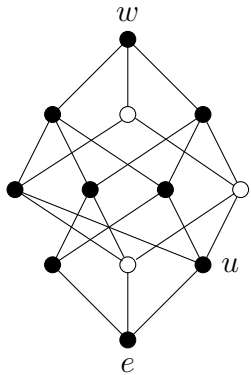

In order to compute $R_{u,w}^{H,x}(q)$ using Theorem~\ref{computa}, we need an $H$-special matching $M$ of $w$. There are 3 of them: we choose, for instance, the dashed $H$-special matching depicted in the first picture in Figure~\ref{accoppiamento}. Hence
$$ R_{u,w}^{H,x}(q)=(q-1-x) R_{u,M(w)}^{H,x}(q).$$
 
Now we need an $H$-special matching $N$ of $M(w)$, and we choose the dashed $H$-special matching depicted in the second picture in Figure~\ref{accoppiamento}. Hence
$$ R_{u,M(w)}^{H,x}(q)= q R_{N(u),NM(w)}^{H,x}(q)+ ( q-1) R_{u,NM(w)}^{H,x}(q)=( q-1) R_{u,NM(w)}^{H,x}(q).$$

Finally,  we need an $H$-special matching of $NM(w)$, and we choose the dashed $H$-special matching depicted in the third picture in Figure~\ref{accoppiamento}. Hence
$$ R_{u,NM(w)}^{H,x}(q)=q  R_{O(u),ONM(w)}^{H,x}(q)+ (q-1)  R_{u,ONM(w)}^{H,x}(q)=(q-1)  R_{u,ONM(w)}^{H,x}(q).$$

Since $u=ONM(w)$ we have $R_{u,ONM(w)}^{H,x}(q)=1$, and the computation yields
$$ R_{u,w}^{H,x}(q)=(q-1-x)(q-1)^2.$$

\begin{figure}[h]
\begin{center}
$$
\begin{tikzpicture}

\draw[fill=black]{(0,0) circle(3pt)};

\draw[fill=black]{(-1,1) circle(3pt)};
\draw[fill=black]{(1,1) circle(3pt)};
\draw[fill=white]{(0,1) circle(3pt)};

\draw[fill=black]{(-1.5,2) circle(3pt)};
\draw[fill=black]{(-0.5,2) circle(3pt)};
\draw[fill=black]{(0.5,2) circle(3pt)};
\draw[fill=white]{(1.5,2) circle(3pt)};

\draw[fill=black]{(-1,3) circle(3pt)};
\draw[fill=black]{(1,3) circle(3pt)};
\draw[fill=white]{(0,3) circle(3pt)};

\draw[fill=black]{(0,4) circle(3pt)};

 \draw{(0,0)--(0,0.9)}; 
 \draw{(-1,1)--(0,0)}; 
  \draw{(1,1)--(0,0)}; 

  \draw{(-1,1)--(-0.5,2)}; 
   \draw{(-1,1)--(0.5,2)}; 
    \draw{(-0.05,1.1)--(-0.5,2)}; 
     \draw{(0.11,1.07)--(1.42,1.95)}; 
      \draw{(-0.07,1.05)--(-1.5,2)};
\draw{(1,1)--(-1.5,2)}; 
\draw{(1,1)--(0.5,2)}; 
\draw{(1,1)--(1.45,1.9)}; 

\draw{(-1.5,2)--(-1,3)}; 
\draw{(-1.5,2)--(-0.08,2.92)}; 
\draw{(-0.5,2)--(-1,3)}; 
\draw{(-0.5,2)--(1,3)}; 
\draw{(0.5,2)--(-1,3)};
\draw{(0.5,2)--(1,3)}; 
\draw{(1.45,2.1)--(1,3)}; 
\draw{(1.415,2.05)--(0.1,2.95)}; 

\draw{(0,4)--(1,3)};
\draw{(0,4)--(0,3.1)};
\draw{(0,4)--(-1,3)};
  
\draw[dashed, line width=3pt]{(0,0)--(0,0.9)};
\draw[dashed, line width=3pt]{(1,1)--(1.45,1.9)};
 \draw[dashed, line width=3pt]{(-1,1)--(-0.5,2)};
 \draw[dashed, line width=3pt]{(1,3)--(0.5,2)};
  \draw[dashed, line width=3pt]{(-1,3)--(0,4)};
  \draw[dashed, line width=3pt]{(-1.5,2)--(-0.1,2.91)};
  
  \draw[fill=black]{(5,0) circle(3pt)};

\draw[fill=black]{(4,1) circle(3pt)};
\draw[fill=black]{(6,1) circle(3pt)};
\draw[fill=white]{(5,1) circle(3pt)};

\draw[fill=black]{(3.5,2) circle(3pt)};
\draw[fill=black]{(4.5,2) circle(3pt)};
\draw[fill=black]{(5.5,2) circle(3pt)};
\draw[fill=white]{(6.5,2) circle(3pt)};

\draw[fill=black]{(4,3) circle(3pt)};
\draw[fill=black]{(6,3) circle(3pt)};
\draw[fill=white]{(5,3) circle(3pt)};

\draw[fill=black]{(5,4) circle(3pt)};

 \draw{(5,0)--(5,0.9)}; 
 \draw{(4,1)--(5,0)}; 
  \draw{(6,1)--(5,0)}; 

  \draw{(4,1)--(4.5,2)}; 
   \draw{(4,1)--(5.5,2)}; 
    \draw{(4.95,1.1)--(4.5,2)}; 
     \draw{(5.11,1.07)--(6.42,1.95)}; 
      \draw{(4.93,1.05)--(3.5,2)};
\draw{(6,1)--(3.5,2)}; 
\draw{(6,1)--(5.5,2)}; 
\draw{(6,1)--(6.45,1.9)}; 

\draw{(3.5,2)--(4,3)}; 
\draw{(3.5,2)--(4.92,2.95)}; 
\draw{(4.5,2)--(4,3)}; 
\draw{(4.5,2)--(6,3)}; 
\draw{(5.5,2)--(4,3)};
\draw{(5.5,2)--(6,3)}; 
\draw{(6.45,2.1)--(6,3)}; 
\draw{(6.415,2.05)--(5.1,2.95)}; 

\draw{(5,4)--(6,3)};
\draw{(5,4)--(5,3.1)};
\draw{(5,4)--(4,3)};

\draw[dashed, line width=3pt]{(4,3)--(5.5,2)};
\draw[dashed, line width=3pt]{(3.5,2)--(6,1)};
\draw[dashed, line width=3pt]{(4.5,2)--(4,1)};
\draw[dashed, line width=3pt]{(5,0)--(5,0.9)};

  \draw[fill=black]{(10,0) circle(3pt)};

\draw[fill=black]{(9,1) circle(3pt)};
\draw[fill=black]{(11,1) circle(3pt)};
\draw[fill=white]{(10,1) circle(3pt)};

\draw[fill=black]{(8.5,2) circle(3pt)};
\draw[fill=black]{(9.5,2) circle(3pt)};
\draw[fill=black]{(10.5,2) circle(3pt)};
\draw[fill=white]{(11.5,2) circle(3pt)};

\draw[fill=black]{(9,3) circle(3pt)};
\draw[fill=black]{(11,3) circle(3pt)};
\draw[fill=white]{(10,3) circle(3pt)};

\draw[fill=black]{(10,4) circle(3pt)};

 \draw{(10,0)--(10,0.9)}; 
 \draw{(9,1)--(10,0)}; 
  \draw{(11,1)--(10,0)}; 

  \draw{(9,1)--(9.5,2)}; 
   \draw{(9,1)--(10.5,2)}; 
    \draw{(9.95,1.1)--(9.5,2)}; 
     \draw{(10.11,1.07)--(11.42,1.95)}; 
      \draw{(9.93,1.05)--(8.5,2)};
\draw{(11,1)--(8.5,2)}; 
\draw{(11,1)--(10.5,2)}; 
\draw{(11,1)--(11.45,1.9)}; 

\draw{(8.5,2)--(9,3)}; 
\draw{(8.5,2)--(9.92,2.95)}; 
\draw{(9.5,2)--(9,3)}; 
\draw{(9.5,2)--(11,3)}; 
\draw{(10.5,2)--(9,3)};
\draw{(10.5,2)--(11,3)}; 
\draw{(11.45,2.1)--(11,3)}; 
\draw{(11.415,2.05)--(10.1,2.95)}; 

\draw{(10,4)--(11,3)};
\draw{(10,4)--(10,3.1)};
\draw{(10,4)--(9,3)};

\draw[dashed, line width=3pt]{(11,1)--(10.5,2)};
\draw[dashed, line width=3pt]{(10,0)--(9,1)};
  
\end{tikzpicture}$$
\end{center}
\caption{\label{accoppiamento} $H$-special matchings.}
\end{figure}
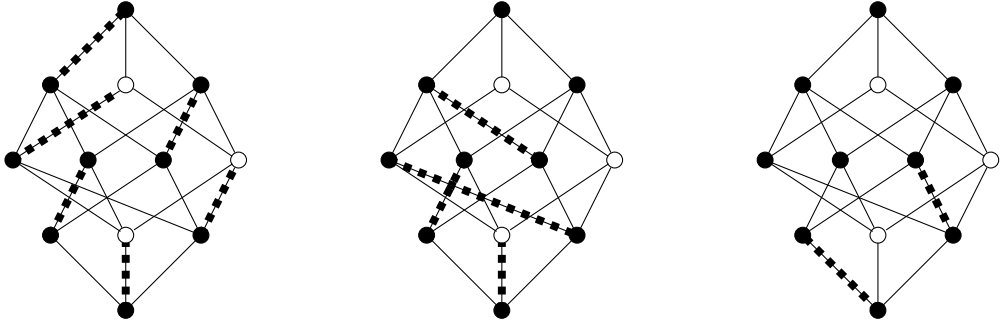

\end{document}